\documentclass[a4paper, 11pt]{article}

\usepackage[utf8]{inputenc}
\usepackage[OT1]{fontenc}
\usepackage{lmodern}     
\usepackage[english]{babel}
\usepackage{cite, hyphenat, hyperref}

\usepackage{amscd, amsfonts, amsmath}
\usepackage{amssymb, amsthm}
\usepackage{bbm, mathrsfs}
\usepackage{enumerate}
\usepackage{calc}
\usepackage[shortlabels]{enumitem}
\usepackage{makeidx}

\usepackage{nccfoots}
\usepackage{appendix}
\usepackage[a4paper, top=4.6cm, bottom=4.5cm, left=3.5cm, right=3.5cm]{geometry}
\allowdisplaybreaks

\numberwithin{equation}{section}
\theoremstyle{plain}

\newtheorem{theorem}{Theorem}
\newtheorem{corollary}[theorem]{Corollary}
\newtheorem{lemma}[theorem]{Lemma}
\newtheorem{proposition}[theorem]{Proposition}

\theoremstyle{definition}

\newtheorem{examples}[theorem]{Examples}

\theoremstyle{remark}
\newtheorem{remark}[theorem]{Remark}

\def\holder{H\"{o}lder }
\def\holders{H\"{o}lder's }

\def\quoteleft{`} 

\begin{document}

\title{On the support of solutions to \\
stochastic differential equations \\
with path-dependent coefficients}
\author{Rama Cont
\footnote{Mathematical Institute, University of Oxford. Email: {\tt Rama.Cont@maths.ox.ac.uk}} 
\and Alexander Kalinin
\setcounter{footnote}{6}\footnote{Department of Mathematics, Imperial College London. Email: {\tt alex.kalinin@mail.de}. Alexander Kalinin's research was supported by a Chapman fellowship from the Department of Mathematics, Imperial College London.}
}
\date{To appear in: \\
\href{https://doi.org/10.1016/j.spa.2019.07.015}{\it Stochastic processes and their applications}, Volume 129 (2019) }
\maketitle

\begin{abstract}
Given a stochastic differential equation with path-dependent coefficients driven by a multidimensional Wiener process, we show that the support of the law of the solution is given by the image of the Cameron-Martin space under the flow of mild solutions to a system of path-dependent ordinary differential equations. Our result extends the Stroock-Varadhan support theorem for diffusion processes to the case of SDEs with path-dependent coefficients. The proof is based on functional It{\^o} calculus.
\end{abstract}


\tableofcontents

\section{Overview}

\subsection{Support theorems for stochastic differential equations}

A stochastic process may be viewed  as a random variable taking values in a space of paths. The (topological) support of this random variable then describes the closure of the set of all attainable paths and provides insight into the structure of sample paths of the process. The nature of the support has been investigated for various classes of stochastic processes, with a focus on stochastic differential equations, under different function space topologies.

For diffusion processes the support with respect to the supremum norm was first described by Stroock and Varadhan~\cite{stroock1972,DiffPro}, a result known as the `Stroock-Varadhan support theorem'. An extension to unbounded coefficients was given by Gy\"{o}ngy~\cite{gyongy1990}. The support of more general Wiener functionals and extensions to SDEs in Hilbert spaces are discussed in Aida et al.~\cite{aida, aida1990}. These results were extended to the \holder topology by Ben Arous et al.~\cite{benarous1994} and, using different techniques, by Millet and Sanz-Sol\'e~\cite{SupportThm}. Bally et al.~\cite{bally1995} use similar methods to derive a support theorem in \holder norm for parabolic SPDEs.
Support theorems in p-variation topology are discussed by Ledoux et al.~\cite{ledoux2002}, using rough path techniques.
Support theorems in \holder and p-variation topologies are discussed in~\cite{frizvictoir} and Pakkanen~\cite{pakkanen2010} gives conditions for a stochastic integral to have full support. In this work we extend some of these results to stochastic differential equations with path-dependent coefficients. 

Let $r,T\geq 0$ with $r < T$ and $d,m\in\mathbb{N}$ and suppose that $(\Omega,\mathscr{F},(\mathscr{F}_{t})_{t\in [0,T]},P)$ is a filtered probability space satisfying the usual conditions and on which there is a standard $d$-dimensional $(\mathscr{F}_{t})_{t\in [0,T]}$-Brownian motion $W$. We consider the following path-dependent stochastic differential equation: 
\begin{equation}\label{SDE}
dX_{t} = b(t,X)\, dt + \sigma(t,X)\,dW_{t}\quad\text{for $t\in [r,T]$}
\end{equation}
with initial condition $X_{s} = \hat{x}(s)$ for $s\in [0,r]$, where $\hat{x}\in C([0,T],\mathbb{R}^{m})$ and the coefficients $b:[r,T]\times C([0,T],\mathbb{R}^{m})\rightarrow\mathbb{R}^{m}$ and $\sigma:[r,T]\times C([0,T],\mathbb{R}^{m})\rightarrow\mathbb{R}^{m\times d}$ are product measurable and non-anticipative in the sense that $b(t,x)$ and $\sigma(t,x)$ depend on the path $x\in C([0,T],\mathbb{R}^{m})$ up to time $t\in [r,T]$ only. 

Under Lipschitz continuity and affine growth conditions on $b$ and $\sigma$, this SDE admits a unique strong solution for which a.e.~sample path lies in the delayed \holder space $C_{r}^\alpha([0,T],\mathbb{R}^{m})$ for every $\alpha\in [0,1/2)$. Our main result is a description of the support of the solution in the \holder topology: we show that the support of its law is given by the image of the Cameron-Martin space under the flow associated with a system of functional differential equations. 

\subsection{Statement of the main result}

Denote by $|\cdot|$   the Euclidean norms in $\mathbb{R}^{d}$ and $\mathbb{R}^{m}$ and the Hilbert-Schmidt norms in $\mathbb{R}^{d\times d}$, $\mathbb{R}^{m\times d}$ and $\mathbb{R}^{m\times m}$. We let $\mathbbm{I}_{d}$ be the $d\times d$ identity matrix and for a matrix $A\in\mathbb{R}^{m\times d}$ we denote by $A'$ its transpose. For a path $x:[0,T]\rightarrow\mathbb{R}^{m}$ we let $x^{t}$ be the path stopped at $t\in [0,T]$: $x^{t}(s)=x(s\wedge t)$ for  $s\in [0,T]$. Throughout, $C([0,T],\mathbb{R}^{m})$ is the separable Banach space 
of all $\mathbb{R}^{m}$-valued continuous maps on $[0,T]$, endowed with the supremum norm given by $\|x\|_{\infty}=\sup_{t\in [0,T]}|x(t)|$. 

For $\alpha\in (0,1]$ we introduce the non-separable Banach space $C_{r}^{\alpha}([0,T],\mathbb{R}^{m})$ of all $x\in C([0,T],\mathbb{R}^{m})$ that are $\alpha$-\holder continuous on $[r,T]$, equipped with the \quoteleft delayed $\alpha$-\holder norm' defined via
\begin{equation}\label{Delayed Hoelder Norm}
\|x\|_{\alpha,r} := \|x^{r}\|_{\infty} + \sup_{s,t\in [r,T]:\,s\neq t}\frac{|x(s) -x(t)|}{|s-t|^{\alpha}}.
\end{equation}
We set $C_{r}^{0}([0,T],\mathbb{R}^{m}):=C([0,T],\mathbb{R}^{m})$ and $\|\cdot\|_{0,r}:=\|\cdot\|_{\infty}$, by convention, and let $H_{r}^{1}([0,T],\mathbb{R}^{m})$ be the separable Banach space of all $x\in C([0,T],\mathbb{R}^{m})$ that are absolutely continuous on $[r,T]$ and whose weak derivative $\dot{x}$ is square-integrable with respect to the Lebesgue measure, endowed with the `delayed Cameron-Martin norm' given by
\begin{equation}\label{Delayed Cameron-Martin Norm}
\|x\|_{H,r} := \|x^{r}\|_{\infty} + \bigg(\int_{r}^{T}|\dot{x}(s)|^{2}\,ds\bigg)^{1/2}.
\end{equation}
Then it holds that $H_{r}^{1}([0,T],\mathbb{R}^{m})\subsetneq C_{r}^{1/2}([0,T],\mathbb{R}^{m})$ and $\|x\|_{1/2,r} \leq \|x\|_{H,r}$ for all $x\in H_{r}^{1}([0,T],\mathbb{R}^{m})$. In the case of no delay, the Cameron-Martin space is a Hilbert space. Further, we allow infinite values and extend the definitions of $\|\cdot\|_{\infty}$ and $\|\cdot\|_{\alpha,r}$ at~\eqref{Delayed Hoelder Norm} to any map $x:[0,T]\rightarrow\mathbb{R}^{m}$ and the definition of $\|\cdot\|_{H,r}$ at~\eqref{Delayed Cameron-Martin Norm} to any map $x:[0,T]\rightarrow\mathbb{R}^{m}$ that is absolutely continuous on $[r,T]$.

We assume the diffusion coefficient $\sigma$ to be of class $\mathbb{C}^{1,2}$ on $[r,T)\times C([0,T],\mathbb{R}^{m})$ in the sense of horizontal and vertical differentiability~\cite{Dupire, CF09}, as discussed in Section~\ref{Non-anticipative functional calculus}, and consider the map $\rho:[r,T]\times C([0,T],\mathbb{R}^{m})\rightarrow\mathbb{R}^{m}$ with components
\begin{equation}\label{Support ODE Correction Term}
\rho_{k}(t,x):= \sum_{l=1}^{d}\partial_{x}\sigma_{k,l}(t,x)\sigma(t,x) e_{l},
\end{equation}
if $t < T$, and $\rho_{k}(t,x):=0$, otherwise. Here, $\{e_{1},\dots,e_{d}\}$ denotes the canonical basis of $\mathbb{R}^{d}$ and $\partial_{x}\sigma_{k,l}:[r,T)\times C([0,T],\mathbb{R}^{m})\rightarrow\mathbb{R}^{1\times m}$ is the vertical derivative of the $(k,l)$-coordinate of $\sigma$ for every $k\in\{1,\dots,m\}$ and $l\in\{1,\dots,d\}$. Note that $\rho = \partial_{x}\sigma \sigma$ for $m=d=1$. Moreover, the horizontal and the second-order vertical derivative of $\sigma$ are denoted by $\partial_{t}\sigma$ and $\partial_{xx}\sigma$, respectively.

We characterize the support of the unique strong solution to~\eqref{SDE} in terms of the following path-dependent ordinary differential equation driven by an element $h\in H_{r}^{1}([0,T],\mathbb{R}^{d})$:
\begin{equation}\label{Support ODE}
\dot{x}_{h}(t) = (b - (1/2)\rho)(t,x_{h})\, + \sigma(t,x_{h})\,\dot{h}(t)\quad\text{for $t\in [r,T]$.}
\end{equation}
Based on these preliminaries, our main result may be stated as follows:

\begin{theorem}[Support theorem for path-dependent SDEs]\label{Support Theorem}\  \\
Let $\sigma$ be of class $\mathbb{C}^{1,2}$ on $[r,T)\times C([0,T],\mathbb{R}^{m})$. Assume that $\sigma$ and $\partial_{x}\sigma$  are bounded and there are $c,\eta,\lambda\geq 0$ and $\kappa\in [0,1)$ such that
\begin{align*}
|b(t,x)|\leq c(1 + \|x\|_{\infty}^{\kappa}),\quad |b(t,x)-b(t,y)|&\leq \lambda\|x-y\|_{\infty},\\
|\partial_{t}\sigma_{k,l}(t,x)| + |\partial_{xx}\sigma_{k,l}(t,x)|&\leq c(1 + \|x\|_{\infty}^{\eta})\quad\text{and}\\
|\sigma(t,x)-\sigma(s,y)| + |\partial_{x}\sigma_{k,l}(t,x) - \partial_{x}\sigma_{k,l}(s,y)|&\leq \lambda(|s-t|^{1/2} + \|x^{t} - y^{s}\|_{\infty})
\end{align*}
for all $s,t\in [r,T)$, $x,y\in C([0,T],\mathbb{R}^{m})$, $k\in\{1,\dots,m\}$ and $l\in\{1,\dots,d\}$. Then the following three assertions are valid:
\begin{enumerate}[(i)]
\item Pathwise uniqueness holds for~\eqref{SDE} and there is a unique strong solution $X$ to~\eqref{SDE} satisfying $X_{s} = \hat{x}(s)$ for all $s\in [0,r]$ a.s. Further, $E[\|X\|_{\alpha,r}^{p}] < \infty$ for each $\alpha\in [0,1/2)$ and $p\geq 1$.
\item For  $h\in H_{r}^{1}([0,T],\mathbb{R}^{d})$ there is a unique mild solution $x_{h}$ to \eqref{Support ODE} such that $x_{h}(s) = \hat{x}(s)$ for all $s\in [0,r]$ and we have $x_{h}\in H_{r}^{1}([0,T],\mathbb{R}^{m})$. In addition, the map $H_{r}^{1}([0,T],\mathbb{R}^{d})\rightarrow H_{r}^{1}([0,T],\mathbb{R}^{m})$, $h\mapsto x_{h}$ is Lipschitz continuous on bounded sets.
\item For $\alpha\in [0,1/2)$ the support of the image measure $P\circ X^{-1}$ in the delayed \holder space $C_{r}^{\alpha}([0,T],\mathbb{R}^m)$ is the closure of the set of all mild solutions $x_{h}$ to~\eqref{Support ODE}, where $h\in H_{r}^{1}([0,T],\mathbb{R}^{d})$. That is,
\begin{equation}\label{Support Equation}
\mathrm{supp}(P\circ X^{-1}) = \overline{\{x_{h}\,|\,h\in H_{r}^{1}([0,T],\mathbb{R}^{d})\}}\quad\text{in $C_{r}^{\alpha}([0,T],\mathbb{R}^{m})$.}
\end{equation}
\end{enumerate}
\end{theorem}

This theorem extends previous results \cite{aida,benarous1994,SupportThm,stroock1972} on the support of diffusion processes to the case of path-dependent coefficients. Moreover, in the diffusion case we retrieve the results of~\cite{benarous1994,SupportThm} under weaker regularity conditions.

Our proof uses the functional It{\^o} calculus~\cite{CF09,cont2012,CF10B, Dupire} to generalize the approach used by Millet and Sanz-Sol\'e \cite{SupportThm} to the path-dependent case, providing the correct Girsanov changes of measures to deduce~\eqref{Support Limit Equality 2}. Based on adapted linear interpolations of Brownian motion, we construct \holder continuous approximations of solutions to~\eqref{SDE} and~\eqref{Support ODE} and show that these approximations converge in \holder norm in probability to the respective solutions. A key ingredient is the functional It{\^o} formula in~\cite{CF09,Dupire,ItoFormula}, combined with interpolation error estimates in supremum norm for stochastic processes.

{\bf Outline}. The remainder of this paper is devoted to the proof of Theorem~\ref{Support Theorem}. Section~\ref{Preliminaries} discusses the various building blocks of the proof. Section~\ref{Non-anticipative functional calculus} recalls several functional calculus concepts from~\cite{CF09,cont2012,CF10B, Dupire} that are useful in our setting. Section~\ref{Mild solutions to path-dependent ODEs} gives conditions for the existence and uniqueness of a mild solution to the path-dependent ODE~\eqref{Support ODE}; Section~\ref{Strong solutions to path-dependent SDEs} gives conditions for the existence of a unique strong solution to the path-dependent SDE~\eqref{SDE}. Section~\ref{Characterization of the support in Hoelder topology} discusses the interpolation method used to characterize the support in the \holder topology.

Section~\ref{Hoelder spaces for stochastic processes} deals with \holder spaces for stochastic processes and the notion of convergence in \holder norm in probability in more depth. Section~\ref{A general Kolmogorov-Chentsov estimate} derives a quantitative version of the Kolmogorov-Chentsov Theorem with an explicit estimate of the \holder norm (Proposition~\ref{Kolmogorov-Chentsov Proposition}). Section~\ref{Convergence along a sequence of partitions} provides a sufficient criterion, based on a sequence of partitions, for convergence in \holder norm in probability (Lemma~\ref{Hoelder Convergence Lemma}). While Section~\ref{Adapted linear interpolations of Brownian motion} discusses adapted linear interpolations of Brownian motion, Section~\ref{Auxiliary convergence results} deduces interpolation error estimates for stochastic processes and several required moment estimates, improving in particular a convergence result from~\cite{SupportThm}[Lemma 3.2].

Section \ref{sec.proofs} proves the existence and uniqueness of mild solutions to path-dependent ODEs (Section~\ref{Proof of Proposition 3}) and strong solutions to path-dependent SDEs (Section~\ref{sec.sdeproof}). Finally, Section~\ref{Proof of the main result} combines these ingredients to give a proof of the main result.

\section{Preliminaries}\label{Preliminaries}

\subsection{Non-anticipative functional calculus}\label{Non-anticipative functional calculus}

Let $D([0,T],\mathbb{R}^{m})$ denote the Banach space of all $\mathbb{R}^{m}$-valued c\`{a}dl\`{a}g maps on $[0,T]$, equipped with the supremum norm $\|\cdot\|_{\infty}$, and recall the following notions from~\cite{cont2012,CF10B}. A functional $G:[r,T)\times D([0,T],\mathbb{R}^{m})\rightarrow\mathbb{R}$ is \emph{non-anticipative} if
\[
\forall t\in [r,T), \  \forall x\in D([0,T],\mathbb{R}^{m}), \quad G(t,x) = G(t,x^{t}),
\]
where $x^{t}=x(t\wedge\cdot)$ is the path $x$ stopped at $t$. Further, $G$ is \emph{boundedness-preserving} if it is  bounded on bounded sets: for each $n\in\mathbb{N}$ there is $c_{n}\geq 0$ such that
\[
\forall t\in [r,T), \  \forall x\in D([0,T],\mathbb{R}^{m}), \quad t\leq Tn/(n+1),\ \|x\|_{\infty}\leq n \quad \Rightarrow \quad |G(t,x)|\leq c_{n}.
\]
We notice that the following pseudometric on $[r,T]\times D([0,T],\mathbb{R}^{m})$ given by
\begin{equation*}
d_{\infty}((t,x),(s,y)):= |t-s|^{1/2} + \|x^{t} - y^{s}\|_{\infty}
\end{equation*}
is complete and if $G$ is $d_{\infty}$-continuous, then it is non-anticipative. As observed in~\cite{ContYi2016}, Lipschitz continuity with respect to $d_{\infty}$ allows for a \holder smoothness of degree $1/2$ in the time variable.

Let us recall the definitions of the horizontal and vertical derivative from~\cite{CF09,Dupire}. A non-anticipative functional $G$ on $[r,T)\times D([0,T],\mathbb{R}^{m})$ is \emph{horizontally differentiable} if for each $t\in [r,T)$ and $x\in D([0,T],\mathbb{R}^{m})$, the function
\[
[0,T-t)\rightarrow\mathbb{R},\quad h\mapsto G(t+h,x^{t})
\]
is differentiable at $0$. In this case, its derivative there is denoted by $\partial_{t}G(t,x)$. We say that $G$ is \emph{vertically differentiable} if for all $t\in [r,T)$ and $x\in D([0,T],\mathbb{R}^{m})$, the function
\[
\mathbb{R}^{m}\rightarrow\mathbb{R},\quad h\mapsto G(t,x+h\mathbbm{1}_{[t,T]})
\]
is differentiable at $0$. If this is the case, then we denote its derivative there by $\partial_{x}G(t,x)$. Hence, $G$ is \emph{partially vertically differentiable} if for every $k\in\{1,\dots,m\}$, $t\in [r,T)$ and $x\in D([0,T],\mathbb{R}^{m})$, the function
\[
\mathbb{R}\rightarrow\mathbb{R},\quad h\mapsto G(t,x+h\overline{e}_{k}\mathbbm{1}_{[t,T]})
\]
is differentiable at $0$, where $\{\overline{e}_{1},\dots,\overline{e}_{m}\}$ is the canonical basis of $\mathbb{R}^{m}$. In this case, its derivative there is denoted by $\partial_{x_{k}}G(t,x)$. If $G$ is vertically differentiable, then it is partially vertically differentiable and $\partial_{x}G=(\partial_{x_{1}}G,\dots,\partial_{x_{m}}G)$.

We call $G$ twice vertically differentiable if it is vertically differentiable and the same holds for $\partial_{x}G$. In this case, we set $\partial_{xx}G := \partial_{x}(\partial_{x}G)$ and
\[
\partial_{x_{k}x_{l}}G:=\partial_{x_{k}}(\partial_{x_{l}}G)\quad\text{for all $k,l\in\{1,\dots,m\}$.}
\]
It follows from Schwarz's Lemma that if $G$ is twice vertically differentiable and $\partial_{xx}G$ is $d_{\infty}$-continuous, then $\partial_{xx}G$ is symmetric: $\partial_{x_{k}x_{l}}G = \partial_{x_{l}x_{k}}G$ for all $k,l\in\{1,\dots,m\}$. Finally, $G$ is \emph{of class $\mathbb{C}^{1,2}$} if it is once horizontally and twice vertically differentiable such that $G$ and its derivatives $\partial_{t}G$, $\partial_{x}G$ and $\partial_{xx}G$ are boundedness-preserving and $d_{\infty}$-continuous.

While horizontal differentiability can be readily extended to functionals that are merely defined on $[r,T)\times C([0,T],\mathbb{R}^{m})$, vertical differentiability is based on c\`{a}dl\`{a}g perturbations, as it involves indicator functions. So, a  functional $F$ on $[r,T)\times C([0,T],\mathbb{R}^{m})$ is said to be of \emph{class $\mathbb{C}^{1,2}$} if it admits an extension $G:[r,T)\times D([0,T],\mathbb{R}^{m})\rightarrow\mathbb{R}$ that satisfies this property. Then it follows from Theorems~5.4.1 and 5.4.2 in~\cite{cont2012} that the definitions
\[
\partial_{x}F :=\partial_{x}G\quad\text{and}\quad\partial_{xx}F:=\partial_{xx}G\quad\text{on $[r,T)\times C([0,T],\mathbb{R}^{m})$}
\]
are independent of the choice of the extension $G$. Based on this uniqueness result, we can apply the functional It{\^o} formula~\cite{ItoFormula}  to prove Proposition~\ref{Main Remainder Proposition 1}, which gives one of the main arguments to characterize the support.

\begin{examples}
(i) Let $\alpha\in (0,1]$ and $\varphi:\mathbb{R}^{m}\rightarrow\mathbb{R}_{+}$ be $\alpha$-\holder continuous. Then the \emph{running supremum functional} $G:[r,T)\times D([0,T],\mathbb{R}^{m})\rightarrow\mathbb{R}$ given by
\[
G(t,x) := \sup_{s\in [0,t]} \varphi(x(s))
\]
is boundedness-preserving and $\alpha$-\holder continuous with respect to $d_{\infty}$. To be more precise, there exists $\lambda \geq 0$ such that
\[
|G(t,x)-G(s,y)|\leq \lambda\|x^{t}-y^{s}\|_{\infty}^{\alpha}
\]
for all $s,t\in [r,T)$ and $x,y\in D([0,T],\mathbb{R}^{m})$. Further, we have $\partial_{t}G=0$. However, if $\varphi$ fails to be differentiable at some $\overline{x}\in\mathbb{R}^{m}$ such that $\varphi(\overline{x} + \overline{h})$ $ > \varphi(\overline{x})$ for all $\overline{h}\in\mathbb{R}^{m}\backslash\{0\}$, then $G$ is not vertically differentiable. For instance, take $\varphi = |\cdot|^{\alpha}$.
\smallskip

\noindent
(ii) Let $\alpha\in (0,1]$, $\varphi:\mathbb{R}^{m}\rightarrow\mathbb{R}^{d}$ be $\alpha$-\holder continuous and $\beta:[r,T)\rightarrow [0,T]$ be right-continuous and satisfy $\beta(t) < t$ for all $t\in (r,T)$. Then the \emph{non-anticipative delayed map} $G:[r,T)\times D([0,T],\mathbb{R}^{m})\rightarrow\mathbb{R}^{d}$ given by
\[
G(t,x):=\varphi((x\circ\beta)(t))
\]
is boundedness-preserving and $\alpha$-\holder continuous in $x\in D([0,T],\mathbb{R}^{m})$, uniformly in $t\in [r,T)$. Unless $\varphi$ is constant, $G$ fails to be horizontally differentiable, yet it is vertically differentiable on $(r,T)\times D([0,T],\mathbb{R}^{m})$ and $\partial_{x}G = 0$ there.
\smallskip

\noindent
(iii) Assume that $\varphi:[0,T)\times\mathbb{R}^{m}\rightarrow\mathbb{R}^{m\times d}$ is a Borel measurable bounded map that is Lipschitz continuous in $\overline{x}\in\mathbb{R}^{m}$, uniformly in $t\in [0,T)$. Then the \emph{non-anticipative integral map} $G:[r,T)\times D([0,T],\mathbb{R}^{m})\rightarrow\mathbb{R}^{m\times d}$ given by
\[
G(t,x) := \int_{0}^{t}\varphi(s,x(s))\,ds
\]
is bounded and $d_{\infty}$-Lipschitz continuous. Moreover, if $\varphi$ is continuous, then $G$ is of class $\mathbb{C}^{1,2}$. In fact, it holds that $\partial_{t}G(t,x) = \varphi(t,x(t))$ for any $t\in [r,T)$ and $x\in D([0,T],\mathbb{R}^{m})$ and $\partial_{x}G = 0$.
\end{examples}

\subsection{Mild solutions to path-dependent ODEs}\label{Mild solutions to path-dependent ODEs}

In this section we show that the ODE~\eqref{Support ODE} admits a unique mild solution which belongs to the delayed Cameron-Martin space $H_{r}^{1}([0,T],\mathbb{R}^{m})$. To this end, let us consider the path-dependent ordinary differential equation:
\begin{equation}\label{General ODE}
\dot{x}(t) = F(t,x) \quad\text{for $t\in [r,T]$,}
\end{equation}
where $F:[r,T]\times C([0,T],\mathbb{R}^{m})\rightarrow\mathbb{R}^{m}$ denotes a non-anticipative product measurable map. Then for each $h\in H_{r}^{1}([0,T],\mathbb{R}^{d})$ the choice $F = b - (1/2)\rho + \sigma\dot{h}$ yields~\eqref{Support ODE} with the correction term $\rho$ given by~\eqref{Support ODE Correction Term}.

In general the map $[r,T]\rightarrow\mathbb{R}^{m}$, $t\mapsto F(t,x)$ may fail to be continuous, so one may not expect solutions in the classical sense. A  \emph{mild solution} to \eqref{General ODE} is a path $x\in C([0,T],\mathbb{R}^{m})$ satisfying
\[
\int_{r}^{T}|F(s,x)|\,ds < \infty\quad\text{and}\quad x(t) = x(r) + \int_{r}^{t}F(s,x)\,ds
\]
for all $t\in [r,T]$. By definition, a mild solution $x$ is absolutely continuous on $[r,T]$ and it becomes a classical solution if and only if the measurable map $[r,T]\rightarrow\mathbb{R}^{m}$, $t\mapsto F(t,x)$ is continuous.

Let us introduce the following regularity conditions, which are satisfied under the assumptions of Theorem~\ref{Support Theorem} for the choice of $F$ mentioned before.

\begin{enumerate}[label=(C.\arabic*), ref=C.\arabic*, leftmargin=\widthof{(C.2)} + \labelsep]
\item\label{C.1} There is a measurable function $c_{0}:[r,T]\rightarrow\mathbb{R}_{+}$ satisfying $\int_{r}^{T}c_{0}(s)^{2}\,ds < \infty$ and
\[
|F(t,x)|\leq c_{0}(t)\bigg(1 + \|x^{r}\|_{\infty} + \int_{r}^{T}|\dot{x}(s)|\,ds\bigg)
\]
for all $t\in [r,T]$ and $x\in C([0,T],\mathbb{R}^{m})$ that is absolutely continuous on $[r,T]$.
\item\label{C.2} For each $n\in\mathbb{N}$ there is a measurable function $\lambda_{n}:[r,T]\rightarrow\mathbb{R}_{+}$ such that $\int_{r}^{T}\lambda_{n}(s)^{2}\,ds < \infty$ and
\[
|F(t,x) - F(t,y)|\leq \lambda_{n}(t)\|x-y\|_{H,r}
\]
for every $t\in [r,T]$ and $x,y\in H_{r}^{1}([0,T],\mathbb{R}^{m})$ with $\|x\|_{H,r}\vee \|y\|_{H,r}\leq n$.
\end{enumerate}

Under the above affine growth condition and Lipschitz smoothness on bounded sets, we obtain a unique mild solution that can be approximated by a Picard iteration in the complete norm $\|\cdot\|_{H,r}$ defined in~\eqref{Delayed Cameron-Martin Norm}.

\begin{proposition}\label{General ODE Proposition}
Let~\eqref{C.1} and~\eqref{C.2} hold. Then~\eqref{General ODE} admits a unique mild solution $x$ satisfying
\[
x(s) = \hat{x}(s)\quad\text{for all $s\in [0,r]$}.
\]
Moreover, $x\in H_{r}^{1}([0,T],\mathbb{R}^{m})$ and the sequence $(x_{n})_{n\in\mathbb{N}_{0}}$ in $H_{r}^{1}([0,T],\mathbb{R}^{m})$, recursively defined by 
\begin{equation*}
x_{0}(t):=\hat{x}(r\wedge t),\qquad
x_{n+1}(t):= x_{0}(t) + \int_{r}^{r\vee t}F(s,x_{n})\,ds
\end{equation*} converges to $x$ in the delayed Cameron-Martin norm $\|\cdot\|_{H,r}$. 
\end{proposition}

\subsection{Strong solutions to path-dependent SDEs}\label{Strong solutions to path-dependent SDEs}

We turn to the derivation of a unique strong solution to the SDE~\eqref{SDE} whose paths lie a.s.~in $C_{r}^{\alpha}([0,T],\mathbb{R}^{m})$ for every $\alpha\in [0,1/2)$. By a \emph{solution} to~\eqref{SDE} we mean an $(\mathscr{F}_{t})_{t\in [0,T]}$-adapted continuous process $X:[0,T]\times\Omega\rightarrow\mathbb{R}^{m}$ satisfying
\begin{align*}
&\int_{r}^{T}|b(s,X)|\,ds + \int_{r}^{T}|\sigma(s,X)|^{2}\,ds < \infty\quad\text{a.s.}\\
&\text{and}\quad X_{t} = X_{r} + \int_{r}^{t}b(s,X)\,ds + \int_{r}^{t}\sigma(s,X)\,dW_{s}
\end{align*}
for every $t\in [r,T]$ a.s. A solution 
 is said to be \emph{strong} if it is  adapted to the augmented natural filtration of the underlying Brownian motion $W$.

We now state the required conditions on the coefficients $b$ and $\sigma$, which are valid in the setting of Theorem~\ref{Support Theorem}.

\begin{enumerate}[label=(C.\arabic*),ref=C.\arabic*, leftmargin=\widthof{(C.4)} + \labelsep]
\setcounter{enumi}{2}
\item\label{C.3} There are a measurable function $c_{0}:[r,T]\rightarrow\mathbb{R}_{+}$ and $\tilde{c}_{0}\geq 0$ such that $\int_{r}^{T}c_{0}(s)^{2}\,ds < \infty$,
\[
|b(t,x)|\leq c_{0}(t)(1 + \|x\|_{\infty})\quad\text{and}\quad |\sigma(t,x)|\leq \tilde{c}_{0}(1 + \|x\|_{\infty})
\]
for all $t\in [r,T]$ and $x\in C([0,T],\mathbb{R}^{m})$.
\item\label{C.4} There are $\alpha_{0}\in [0,1/2)$, a measurable function $\lambda_{0}:[r,T]\rightarrow\mathbb{R}_{+}$ and $\tilde{\lambda}_{0}\geq 0$ such that $\int_{r}^{T}\lambda_{0}(s)^{2}\,ds < \infty$,
\begin{align*}
|b(t,x) - b(t,y)|\leq \lambda_{0}(t)\|x-y\|_{\alpha_{0},r}\quad\text{and}\quad |\sigma(t,x) - \sigma(t,y)|\leq\tilde{\lambda}_{0}\|x-y\|_{\alpha_{0},r}
\end{align*}
for each $t\in [r,T]$ and $x,y\in C_{r}^{\alpha_{0}}([0,T],\mathbb{R}^{m})$.
\end{enumerate}

\begin{remark}
If~\eqref{C.4} holds, then it is also true if $\alpha_{0}$ is replaced by any $\alpha\in (\alpha_{0},1]$. Thus, it is strongest in the case that $\alpha_{0}=0$, since $\|\cdot\|_{0}=\|\cdot\|_{\infty}$, by convention.
\end{remark}

Let $\alpha\in [0,1]$ and $p\geq 1$. By Proposition~\ref{Hoelder Completeness Proposition 1}, the space $\mathscr{C}_{r,p}^{\alpha}([0,T],\mathbb{R}^{m})$  of all $(\mathscr{F}_{t})_{t\in [0,T]}$-adapted continuous processes $X:[0,T]\times\Omega\rightarrow\mathbb{R}^{m}$ for which $E[\|X\|_{\alpha,r}^{p}]$ is finite, equipped with the seminorm
\begin{equation}\label{General SDE Seminorm}
\mathscr{C}_{r,p}^{\alpha}([0,T],\mathbb{R}^{m})\rightarrow\mathbb{R}_{+},\quad X\mapsto \big(E\big[\|X\|_{\alpha,r}^{p}\big]\big)^{1/p}
\end{equation}
is complete. Note that if a sequence $(\null_{n}X)_{n\in\mathbb{N}}$ in this seminormed space converges, then it also converges in the norm $\|\cdot\|_{\alpha,r}$ in probability. We define
$$\mathscr{C}_{r,\infty}^{1/2-}([0,T],\mathbb{R}^{m}) :=\mathop{\bigcap}_{\alpha\in [0,1/2),\, p\geq 1} \mathscr{C}_{r,p}^{\alpha}([0,T],\mathbb{R}^{m}),
$$
which is a completely pseudometrizable topological space.

\begin{proposition}\label{General SDE Proposition}
Let~\eqref{C.3} and~\eqref{C.4} be valid. Then, up to indistinguishability, there is a unique strong solution $X$ to~\eqref{SDE} such that
\[
X_{s} = \hat{x}(s)\quad\text{for all $s\in [0,r]$ a.s.}
\]
and it holds that $X\in\mathscr{C}_{r,\infty}^{1/2-}([0,T],\mathbb{R}^{m})$. Furthermore, the sequence $(\null_{n}X)_{n\in\mathbb{N}_{0}}$ in $\mathscr{C}_{r,\infty}^{1/2-}([0,T],\mathbb{R}^{m})$, recursively given by 
\begin{equation*}
\begin{split}\null_{0}X_{t} := \hat{x}(r\wedge t),\qquad
\null_{n+1}X_{t}&= \null_{0}X_{t} + \int_{r}^{r\vee t}b(s,\null_{n}X)\,ds + \int_{r}^{r\vee t}\sigma(s,\null_{n}X)\,dW_{s}
\end{split}
\end{equation*}
converges to $X$ in the norm $\|\cdot\|_{\alpha,r}$ in $p$-th moment for each $\alpha\in [0,1/2)$ and $p\geq 1$:
\[
\lim_{n\uparrow\infty}E[\|\null_{n}X - X\|_{\alpha,r}^{p}]= 0.
\]
\end{proposition}

\begin{remark}
Pathwise uniqueness for~\eqref{SDE} is shown in Lemma~\ref{General SDE Lemma 2}, requiring only the following Lipschitz condition on bounded sets, which follows from~\eqref{C.4} in the strongest case $\alpha_{0} =0$:
\begin{enumerate}[label=(C.\arabic*),ref=C.\arabic*,leftmargin=\widthof{(C.5)} + \labelsep]
\setcounter{enumi}{4}
\item\label{C.5} For each $n\in\mathbb{N}$ there is a measurable function $\lambda_{n}:[r,T]\rightarrow\mathbb{R}_{+}$ satisfying $\int_{r}^{T}\lambda_{n}(s)^{2}\,ds < \infty$ and
\[
|b(t,x) - b(t,y)| + |\sigma(t,x) - \sigma(t,y)|\leq \lambda_{n}(t)\|x-y\|_{\infty}
\]
for every $t\in [r,T]$ and $x,y\in C([0,T],\mathbb{R}^{m})$ with $\|x\|_{\infty}\vee\|y\|_{\infty}\leq n$.
\end{enumerate}
\end{remark}

\subsection{Characterization of the support in \holder topology}\label{Characterization of the support in Hoelder topology}

Sections~\ref{Mild solutions to path-dependent ODEs} and~\ref{Strong solutions to path-dependent SDEs} provide the main arguments to prove the first two assertions of Theorem~\ref{Support Theorem}. Let us now describe how we will prove the characterization~\eqref{Support Equation} of the support. For $n\in\mathbb{N}$ let $\mathbb{T}_{n}$ be a partition of $[r,T]$ that we write in the form
\[
\mathbb{T}_{n}=\{t_{0,n},\dots,t_{k_{n,n}}\}
\]
for some $k_{n}\in\mathbb{N}$ and $t_{0,n},\dots,t_{k_{n},n}\in [r,T]$ such that $r=t_{0,n} < \cdots < t_{k_{n},n} = T$ and we denote its mesh by $|\mathbb{T}_{n}|=\max_{i\in\{0,\dots,k_{n}-1\}} (t_{i+1,n} - t_{i,n})$. We assume that $\lim_{n\uparrow\infty} |\mathbb{T}_{n}| = 0$ and that the sequence $(\mathbb{T}_{n})_{n\in\mathbb{N}}$ of partitions is \emph{balanced} in the sense of~\cite{das17}. That is, there is $c_{\mathbb{T}}\geq 1$ such that
\begin{equation}\label{Partition Condition}
|\mathbb{T}_{n}|\leq c_{\mathbb{T}} \min_{i\in\{0,\dots,k_{n}-1\}} (t_{i+1,n} - t_{i,n}) \quad\text{for all $n\in\mathbb{N}$.}
\end{equation}
For $k,n\in\mathbb{N}$ and every path $x:[0,T]\rightarrow\mathbb{R}^{k}$ we let $L_{n}(x):[0,T]\rightarrow\mathbb{R}^{k}$ denote the unique map satisfying $L_{n}(x)(t)=x(r\wedge t)$ for $t\in [0,t_{1,n}]$ and
\begin{equation}\label{Delayed Linear Interpolation Operator}
L_{n}(x)(t)= x(t_{i-1,n}) + \frac{t-t_{i,n}}{t_{i+1,n}-t_{i,n}}(x(t_{i,n}) - x(t_{i-1,n}))
\end{equation}
for $i\in\{1,\dots,k_{n}-1\}$ and $t\in [t_{i,n},t_{i+1,n}]$. Then $L_{n}(x)$ is piecewise continuously differentiable on $[r,T]$ and can be regarded as \emph{delayed linear interpolation} of $x$ along $\mathbb{T}_{n}$ on this interval. Hence, we may define an adapted process $\null_{n}W:[0,T]\times\Omega\rightarrow\mathbb{R}^{d}$, whose paths belong to $H_{r}^{1}([0,T],\mathbb{R}^{d})$, by setting
\begin{equation}\label{Adapted Linear Interpolation of Brownian Motion Definition}
\null_{n}W_{t} := L_{n}(W)(t).
\end{equation}

Let us for the moment suppose that the assumptions and the first two claims of Theorem~\ref{Support Theorem} hold. To establish that $\mathrm{supp}(P\circ X^{-1})$ is included in the closure of $\{x_{h}\,|\,h\in H_{r}^{1}([0,T],\mathbb{R}^{d})\}$ in $C_{r}^{\alpha}([0,T],\mathbb{R}^{m})$ for $\alpha\in [0,1/2)$, we will justify in Section~\ref{Proofs of Theorem 7 and 1} that it suffices to check that
\begin{equation}\label{Support Limit Equality 1}
\lim_{n\uparrow\infty} P(\|x_{\null_{n}W} - X\|_{\alpha,r}\geq \varepsilon) =0\quad\text{for all $\varepsilon > 0$.}
\end{equation}
We remark that, by the definition of a mild solution to~\eqref{General ODE}, for each $n\in\mathbb{N}$ the adapted process $x_{\null_{n} W}$, whose paths lie in $H_{r}^{1}([0,T],\mathbb{R}^{m})$, is a strong solution to the degenerate SDE
\[
d\null_{n}Y_{t} = \big((b - (1/2)\rho)(t,\null_{n}Y) + \sigma(t,\null_{n}Y)\null_{n}\dot{W}_{t}\big)\,dt\quad\text{for $t\in [r,T]$}
\]
with initial condition $\null_{n}Y^{r} = \hat{x}^{r}$. Next, we will show that the converse inclusion in~\eqref{Support Equation} follows if for every $h\in H_{r}^{1}([0,T],\mathbb{R}^{d})$ we can find a sequence $(P_{h,n})_{n\in\mathbb{N}}$ of probability measures on $(\Omega,\mathscr{F})$ that are absolutely continuous to $P$ such that
\begin{equation}\label{Support Limit Equality 2}
\lim_{n\uparrow\infty} P_{h,n}(\|X - x_{h}\|_{\alpha,r}\geq\varepsilon) = 0\quad\text{for any $\varepsilon > 0$.}
\end{equation}
We emphasize that for each $n\in\mathbb{N}$ the measure $P_{h,n}$ will be constructed by means of Girsanov's theorem such that $X$ is a strong solution to the SDE
\begin{equation}\label{Girsanov SDE}
d\null_{n}Y_{t} = \big(b(t,\null_{n}Y) + \sigma(t,\null_{n}Y)(\dot{h} - \dot{L}_{n}(\null_{h,n}W))(t)\big)\,dt + \sigma(t,\null_{n}Y)\,d\null_{h,n}W_{t}
\end{equation}
for $t\in [r,T]$ under $P_{h,n}$ with initial condition $\null_{n}Y^{r} = \hat{x}^{r}$, where $\null_{h,n}W$ is some standard $d$-dimensional $(\mathscr{F}_{t})_{t\in [0,T]}$-Brownian motion under $P_{h,n}$. Hence, to prove \eqref{Support Limit Equality 1} and \eqref{Support Limit Equality 2} at the same time, we consider the following general framework.

Assume   $\underline{B}$ is an $\mathbb{R}^{m}$-valued and $B_{H}$, $\overline{B}$ and $\Sigma$ are $\mathbb{R}^{m\times d}$-valued non-anticipative product measurable maps on $[r,T]\times C([0,T],\mathbb{R}^{m})$. Then for each $n\in\mathbb{N}$ we introduce the SDE
\begin{equation}\label{General Sequence SDE}
\begin{split}
d\null_{n}Y_{t} &= \big(\underline{B}(t,\null_{n}Y) + B_{H}(t,\null_{n}Y)\dot{h}(t) + \overline{B}(t,\null_{n}Y)\null_{n}\dot{W}_{t}\big)\,dt + \Sigma(t,\null_{n}Y)\,dW_{t}
\end{split}
\end{equation}
for $t\in [r,T]$. Under the hypothesis that $\overline{B}$ is of class $\mathbb{C}^{1,2}$ on $[r,T)\times C([0,T],\mathbb{R}^{m})$, we also consider the SDE
\begin{equation}\label{General Limit SDE}
\begin{split}
dY_{t} &= \big((\underline{B} + R)(t,Y) + B_{H}(t,Y)\dot{h}(t)\big)\,dt + (\overline{B}+\Sigma)(t,Y)\,dW_{t}
\end{split}
\end{equation}
for $t\in [r,T]$, where the $\mathbb{R}^{m}$-valued non-anticipative product measurable map $R$ on $[r,T]\times C([0,T],\mathbb{R}^{m})$ is given coordinatewise by
\begin{equation}\label{General Remainder Term}
R_{k}(t,x) := \sum_{l=1}^{d}\partial_{x}\overline{B}_{k,l}(t,x)((1/2)\overline{B} + \Sigma)(t,x)e_{l},
\end{equation}
if $t < T$, and $R_{k}(t,x):=0$, otherwise. In Theorem~\ref{General SDE Convergence Theorem} we show that whenever $\null_{n}Y$ and $Y$ are two solutions to~\eqref{General Sequence SDE} and~\eqref{General Limit SDE}, respectively, such that $\null_{n}Y^{r}=Y^{r}=\hat{x}^{r}$ a.s.~for all $n\in\mathbb{N}$, then
\begin{equation}\label{General SDE Hoelder Limit Equation}
\lim_{n\uparrow\infty} P(\|\null_{n}Y - Y\|_{\alpha,r}\geq\varepsilon) = 0\quad\text{for each $\varepsilon > 0$.}
\end{equation}
{Then the choice $\underline{B} = b - (1/2)\rho$, $B_{H} = 0$, $\overline{B}=\sigma$ and $\Sigma = 0$ yields~\eqref{Support Limit Equality 1}, since $R = (1/2)\rho$ in this case. Moreover, by choosing $\underline{B} = b$, $B_{H} = \sigma$, $\overline{B}=-\sigma$ and $\Sigma = \sigma$ instead,~\eqref{Support Limit Equality 2} follows. Since these are the two desired results, let us consider the following required conditions:}
\begin{enumerate}[label=(C.\arabic*), ref=C.\arabic*, leftmargin=\widthof{(8)} + \labelsep]
\setcounter{enumi}{5}
\item\label{C.6} $\overline{B}$ is of class $\mathbb{C}^{1,2}$ on $[r,T)\times C([0,T],\mathbb{R}^{m})$ and there are $c,\eta\geq 0$ and $\kappa\in [0,1)$ such that $|\underline{B}(t,x)| + |B_{H}(t,x)|\leq c(1 + \|x\|_{\infty}^{\kappa})$,
\begin{align*}
|\overline{B}(t,x)| + \bigg(\sum_{k=1}^{m}\sum_{l=1}^{d}|\partial_{x}\overline{B}_{k,l}(t,x)|^{2}\bigg)^{1/2} + |\Sigma(t,x)|&\leq c\quad\text{and}\\
|\partial_{t}\overline{B}(t,x)| + \bigg(\sum_{k=1}^{m}\sum_{l=1}^{d}|\partial_{xx}\overline{B}_{k,l}(t,x)|^{2}\bigg)^{1/2}&\leq c(1 + \|x\|_{\infty}^{\eta})
\end{align*}
for all $t\in [r,T)$ and $x\in C([0,T],\mathbb{R}^{m})$.

\item\label{C.7} $\underline{B}$ is Lipschitz continuous in $x\in C([0,T],\mathbb{R}^{m})$, uniformly in $t\in [r,T)$, and $B_{H}$, $\overline{B}$, $\partial_{x}\overline{B}$ and $\Sigma$ are $d_{\infty}$-Lipschitz continuous.

\item\label{C.8} There are $\overline{b}_{0}\in\mathbb{R}$ and a measurable function $\overline{b}:[r,T]\rightarrow\mathbb{R}$ such that $\int_{r}^{T}|\overline{b}(s)|^{2}\,ds$ $ < \infty$ and $\overline{b}_{0}\overline{B}(t,x) = \overline{b}(t)\Sigma(t,x)$ for all $t\in [r,T)$ and $x\in C([0,T],\mathbb{R}^{m})$.
\end{enumerate}

\begin{theorem}\label{General SDE Convergence Theorem}
Let~\eqref{C.6}-\eqref{C.8} be satisfied, $h\in H_{r}^{1}([0,T],\mathbb{R}^{d})$ and $(x_{n})_{n\in\mathbb{N}}$ be a bounded sequence in $C([0,T],\mathbb{R}^{m})$. Then the following three assertions hold:
\begin{enumerate}[(i)]
\item For any $n\in\mathbb{N}$ there is a unique strong solution $\null_{n}Y$ to~\eqref{General Sequence SDE} such that $\null_{n}Y^{r}$ $= x_{n}^{r}$ a.s. Moreover, $\sup_{n\in\mathbb{N}}E[\|\null_{n}Y\|_{\alpha,r}^{p}] < \infty$ for all $\alpha\in [0,1/2)$ and $p\geq 1$.
\item There is a unique strong solution $Y$ to~\eqref{General Limit SDE} such that $Y^{r} = \hat{x}^{r}$ a.s. and we have $E[\|Y\|_{\alpha,r}^{p}] < \infty$ for every $\alpha\in [0,1/2)$ and $p\geq 1$.
\item Let $\alpha\in [0,1/2)$ satisfy $\lim_{n\uparrow\infty}\|x_{n}^{r} - \hat{x}^{r}\|_{\infty}/|\mathbb{T}_{n}|^{\alpha}= 0$. Then
\begin{equation}\label{General SDE Partition Limit Equation}
\lim_{n\uparrow\infty}E\big[\max_{j\in\{0,\dots,k_{n}\}}|\null_{n}Y_{t_{j,n}} - Y_{t_{j,n}}|^{2}\big]/|\mathbb{T}_{n}|^{2\alpha} = 0.
\end{equation}
In particular,~\eqref{General SDE Hoelder Limit Equation} holds: $(\null_{n}Y)_{n\in\mathbb{N}}$ converges in the norm $\|\cdot\|_{\alpha,r}$ in probability to $Y$.
\end{enumerate}
\end{theorem}

\begin{remark}
{Condition~\eqref{C.8} allows us to perform a change of measure to obtain a unique strong solution to~\eqref{General Sequence SDE}. However, when deriving~\eqref{General SDE Partition Limit Equation} in Sections~\ref{Decomposition into remainder terms},~\ref{Convergence of the first two remainders} and~\ref{Convergence of the third remainder}, we merely assume~\eqref{C.6}~and~\eqref{C.7}.}
\end{remark}

\section{Convergence in \holder norm in probability}\label{Convergence in Hoelder norm in probability}

\subsection{H\"{o}lder spaces for stochastic processes}\label{Hoelder spaces for stochastic processes}

For $\alpha\in [0,1]$ we let $\mathscr{C}_{r}^{\alpha}([0,T],\mathbb{R}^{m})$ denote the linear space of all adapted continuous processes $X:[0,T]\times\Omega\rightarrow\mathbb{R}^{m}$ satisfying $X\in C_{r}^{\alpha}([0,T],\mathbb{R}^{m})$ a.s., endowed with the pseudometric
\begin{equation}\label{Hoelder Probability Pseudometric}
\begin{split}
&\mathscr{C}_{r}^{\alpha}([0,T],\mathbb{R}^{m})\times \mathscr{C}_{r}^{\alpha}([0,T],\mathbb{R}^{m})\rightarrow\mathbb{R}_{+},\\
&(X,Y)\mapsto E[\|X - Y\|_{\alpha,r}\wedge 1].
\end{split}
\end{equation}
Then a sequence $(\null_{n}X)_{n\in\mathbb{N}}$ in this space converges to some $X\in\mathscr{C}_{r}^{\alpha}([0,T],\mathbb{R}^{m})$ if and only if it converges in the norm $\|\cdot\|_{\alpha,r}$ in probability to this process. Put differently, $(\|\null_{n}X - X\|_{\alpha,r})_{n\in\mathbb{N}}$ converges in probability to zero. Further, $(\null_{n}X)_{n\in\mathbb{N}}$ is Cauchy if and only if
\[
\lim_{k\uparrow\infty}\sup_{n\in\mathbb{N}:\,n\geq k} P(\|\null_{n}X - \null_{k}X\|_{\alpha,r}\geq\varepsilon) = 0\quad\text{for all $\varepsilon > 0$.}
\]
To abbreviate notation, we set $\mathscr{C}([0,T],\mathbb{R}^{m}):=\mathscr{C}_{r}^{0}([0,T],\mathbb{R}^{m})$, which represents the linear space of all $\mathbb{R}^{m}$-valued adapted continuous processes. The fact that $\mathscr{C}([0,T],\mathbb{R}^{m})$ is complete can be extended, as the following result shows.

\begin{lemma}\label{Hoelder Completeness Lemma 1}
The linear space $\mathscr{C}_{r}^{\alpha}([0,T],\mathbb{R}^{m})$ endowed with the pseudometric~\eqref{Hoelder Probability Pseudometric} is complete.
\end{lemma}

\begin{proof}
Let $(\null_{n}X)_{n\in\mathbb{N}}$ be a Cauchy sequence in $\mathscr{C}_{r}^{\alpha}([0,T],\mathbb{R}^{m})$. By Lemma 4.3.3 in~\cite{DiffPro} for instance, there is $X\in\mathscr{C}([0,T],\mathbb{R}^{m})$ to which $(\null_{n}X)_{n\in\mathbb{N}}$ converges uniformly in probability. For given $\varepsilon, \eta > 0$ there is $n_{0}\in\mathbb{N}$ such that
\[
P\bigg(\sup_{s,t\in [r,T]:\,s\neq t} \frac{|(\null_{n}X_{s} - \null_{k}X_{s}) - (\null_{n}X_{t} - \null_{k}X_{t})|}{|s-t|^{\alpha}} \geq \frac{\varepsilon}{2}\bigg) < \frac{\eta}{2}
\]
for all $k,n\in\mathbb{N}$ with $k\wedge n\geq n_{0}$. We fix $l\in\mathbb{N}$ and set $\delta_{l}:=(T-r)/l$, then there exists $k_{l}\in\mathbb{N}$ such that $k_{l}\geq n_{0}$ and $P(\|\null_{k_{l}}X-X\|_{\infty}\geq (\varepsilon/4)\delta_{l}^{\alpha}) < \eta/2$. Hence,
\[
P\bigg(\sup_{s,t\in [r,T]:\, |s-t|\geq \delta_{l}} \frac{|(\null_{n}X_{s} - X_{s}) - (\null_{n}X_{t} - X_{t})|}{|s-t|^{\alpha}} > \varepsilon\bigg) < \eta
\]
for each $n\in\mathbb{N}$ with $n\geq n_{0}$. By the continuity of measures, $(\|\null_{n}X-X\|_{\alpha,r})_{n\in\mathbb{N}}$ converges in probability to zero. In particular, $\|X\|_{\alpha,r} < \infty$ a.s.
\end{proof}

For $p\geq 1$ we recall the linear space $\mathscr{C}_{r,p}^{\alpha}([0,T],\mathbb{R}^{m})$ of all $X\in\mathscr{C}_{r}^{\alpha}([0,T],\mathbb{R}^{m})$ for which $E[\|X\|_{\alpha,r}^{p}]$ is finite, endowed with the seminorm~\eqref{General SDE Seminorm}. We say that a sequence $(\null_{n}X)_{n\in\mathbb{N}}$ in this space is \emph{$p$-fold uniformly integrable} if $(\|\null_{n}X\|_{\alpha,r})_{n\in\mathbb{N}}$ satisfies this property in the usual sense.

\begin{lemma}\label{Hoelder Completeness Lemma 2}
Any Cauchy sequence $(\null_{n}X)_{n\in\mathbb{N}}$ in $\mathscr{C}_{r,p}^{\alpha}([0,T],\mathbb{R}^{m})$ is $p$-fold uniformly integrable.
\end{lemma}

\begin{proof}
Let $\varepsilon > 0$, then there is $n_{0}\in\mathbb{N}$ such that $E[\|\null_{n}X-\null_{k}X\|_{\alpha,r}^{p}] < \varepsilon/2^{p}$ for all $k,n\in\mathbb{N}$ with $k\wedge n\geq n_{0}$. As the random variable $Y:=\max_{n\in\{1,\dots,n_{0}\}}\|\null_{n}X\|_{\alpha,r}$ is $p$-fold integrable, we obtain that
\[
\sup_{n\in\mathbb{N}}\big(E\big[\mathbbm{1}_{A}\|\null_{n}X\|_{\alpha,r}^{p}\big]\big)^{1/p}\leq (E[\mathbbm{1}_{A}Y^{p}])^{1/p} + \varepsilon^{1/p}/2
\]
for each $A\in\mathscr{F}$. First, by choosing $A=\Omega$, this gives $\sup_{n\in\mathbb{N}} E[\|\null_{n}X\|_{\alpha,r}^{p}] < \infty$. Secondly, by setting $\delta:=\varepsilon/2^{p}$, it follows that $\sup_{n\in\mathbb{N}} E[\mathbbm{1}_{A}\|\null_{n}X\|_{\alpha,r}^{p}] < \varepsilon$ for every $A\in\mathscr{F}$ with $E[\mathbbm{1}_{A}Y^{p}] < \delta$.
\end{proof}

We conclude with the following convergence characterization.

\begin{proposition}\label{Hoelder Completeness Proposition 1}
A sequence $(\null_{n}X)_{n\in\mathbb{N}}$ in $\mathscr{C}_{r,p}^{\alpha}([0,T],\mathbb{R}^{m})$ converges with respect to the seminorm~\eqref{General SDE Seminorm} if and only if it is $p$-fold uniformly integrable and there is $X\in\mathscr{C}_{r}^{\alpha}([0,T],\mathbb{R}^{m})$ such that
\[
\lim_{n\uparrow\infty} P(\|\null_{n}X - X\|_{\alpha,r}\geq\varepsilon) =0\quad\text{for all $\varepsilon > 0$.}
\]
In the latter case, $E[\|X\|_{\alpha,r}^{p}] < \infty$ and $\lim_{n\uparrow\infty}E[\|\null_{n}X-X\|_{\alpha,r}^{p}]=0$. Moreover, $\mathscr{C}_{r,p}^{\alpha}([0,T],\mathbb{R}^{m})$ equipped with~\eqref{General SDE Seminorm} is complete.
\end{proposition}

\begin{proof}
By Lemmas~\ref{Hoelder Completeness Lemma 1} and~\ref{Hoelder Completeness Lemma 2}, it suffices to show the if-direction of the first claim. We let $(\nu_{n})_{n\in\mathbb{N}}$ be a strictly increasing sequence in $\mathbb{N}$ such that $(\|\null_{\nu_{n}}X-X\|_{\alpha,r})_{n\in\mathbb{N}}$ converges a.s.~to zero, then
\[
E\big[\|X\|_{\alpha,r}^{p}\big]\leq \liminf_{n\uparrow\infty}E\big[\|\null_{\nu_{n}}X\|_{\alpha,r}^{p}\big]\leq \sup_{n\in\mathbb{N}} E\big[\|\null_{n}X\|_{\alpha,r}^{p}\big] <\infty,
\]
by Fatou's Lemma. For $\varepsilon > 0$ there are $\delta > 0$ and $n_{0}\in\mathbb{N}$ such that $E[\mathbbm{1}_{A}\|\null_{n}X\|_{\alpha,r}^{p}]$ $< (\varepsilon/3)^{p}$ and $P(\|\null_{n}X - X\|_{\alpha,r}\geq\varepsilon/3)$ $< \delta$ for every $A\in\mathscr{F}$ and $n\in\mathbb{N}$ with $P(A) < \delta$ and $n\geq n_{0}$. Thus,
\[
\big(E\big[\|\null_{n}X - X\|_{\alpha,r}^{p}\big]\big)^{1/p}\leq \big(E\big[\mathbbm{1}_{\{\|\null_{n}X - X\|_{\alpha,r}\geq\varepsilon/3\}}\|\null_{n}X - X\|_{\alpha,r}^{p}\big]\big)^{1/p} + \varepsilon/3 < \varepsilon
\]
for any such $n\in\mathbb{N}$, because similar reasoning as before gives us that $E[\mathbbm{1}_{A}\|X\|_{\alpha,r}^{p}]$ $\leq \sup_{n\in\mathbb{N}}E[\mathbbm{1}_{A}\|\null_{n}X\|_{\alpha,r}^{p}]$ $ < \infty$ for all $A\in\mathscr{F}$. This completes the proof.
\end{proof}

\subsection{A general Kolmogorov-Chentsov estimate}\label{A general Kolmogorov-Chentsov estimate}

We revisit the proof of the Kolmogorov-Chentsov Theorem (see e.g.~Theorem 2.1 in~\cite{RevuzYor}) to obtain a quantitative estimate of the \holder norm. Let
\begin{equation}\label{Kolmogorov-Chentsov Constant}
k_{\alpha,p,q}:=2^{p + q}(2^{q/p -\alpha}-1)^{-p}
\end{equation}
for $p\geq 1$, $q > 0$ and $\alpha\in [0,q/p)$ and note that the function $[0,q/p)\rightarrow (2^{p},\infty)$, $\alpha\mapsto k_{\alpha,p,q}$ is strictly increasing with $\lim_{\alpha\uparrow q/p} k_{\alpha,p,q}=\infty$. Then the following result holds, in which the process in question and not necessarily a modification appears.

\begin{proposition}\label{Kolmogorov-Chentsov Proposition}
Assume that $X$ is an $\mathbb{R}^{m}$-valued right-continuous process for which there are $c_{0}\geq 0$, $p\geq 1$ and $q > 0$ such that
\begin{equation}\label{Kolmogorov-Chentsov Condition}
E[|X_{s} - X_{t}|^{p}] \leq c_{0}|s-t|^{1+q}
\end{equation}
for all $s,t\in [r,T]$. Then for each $\alpha\in [0,q/p)$ it holds that
\[
E\bigg[\sup_{s,t\in [r,T]:\,s\neq t}\frac{|X_{s} - X_{t}|^{p}}{|s-t|^{\alpha p}}\bigg] \leq  k_{\alpha,p,q} c_{0}(T - r)^{1 + q - \alpha p}.
\]
In particular if $q\leq p$, then the sample paths of $X$ are a.s.~$\alpha$-\holder continuous  on $[r,T]$ for all $\alpha\in [0,q/p)$.
\end{proposition}

\begin{proof}
For given $n\in\mathbb{N}_{0}$ let $\mathbb{D}_{n}$ be the $n$-th dyadic partition of $[r,T]$ whose points are $d_{i,n}:= r + i2^{-n}(T-r)$, where $i\in\{0,\dots,2^{n}\}$. We define
\[
\Delta_{n}:=\{(s,t)\in\mathbb{D}_{n}\times\mathbb{D}_{n}\,|\,|s-t| \leq 2^{-n}(T-r)\},
\]
then it is readily seen that there are $2^{n}$ tuples $(s,t)\in\Delta_{n}$ satisfying $s < t$. For $Y_{n}:=\sup_{(s,t)\in\Delta_{n}} |X_{s} - X_{t}|$ condition \eqref{Kolmogorov-Chentsov Condition} gives
\begin{equation}\label{Hoelder Moment Condition I}
E[Y_{n}^{p}]\leq\sum_{(s,t)\in\Delta_{n}:\, s < t}E[|X_{s} - X_{t}|^{p}]\leq 2^{-nq}c_{0}(T-r)^{1+q}.
\end{equation}

We set $\mathbb{D}:=\bigcup_{n\in\mathbb{N}_{0}}\mathbb{D}_{n}$ and let $s,t\in\mathbb{D}$ satisfy $0 < t - s$ $< 2^{-n}(T-r)$ for some $n\in\mathbb{N}_{0}$. Then for each $k\in\mathbb{N}_{0}$ there are unique $i_{k},j_{k}\in\{1,\dots,2^{k}\}$ such that $d_{i_{k} - 1,k} \leq s < d_{i_{k},k}$, and
\[
d_{j_{k}-1,k}\leq t < d_{j_{k},k},\quad\text{if $t < T$,}\quad \text{and}\quad d_{j_{k},k} = T,\quad\text{otherwise.}
\]
As $(d_{i_{k},k})_{k\in\mathbb{N}_{0}}$ and $(d_{j_{k},k})_{k\in\mathbb{N}_{0}}$ are two decreasing sequences converging to $s$ and $t$, respectively, two telescoping sums yield that
\begin{align*}
X_{s} - X_{t} &=  X_{d_{i_{n},n}} - X_{d_{j_{n},n}} + \sum_{k=n}^{\infty} \big( X_{d_{i_{k+1},k+1}} - X_{d_{i_{k},k}}\big) + \sum_{k=n}^{\infty}\big(X_{d_{j_{k+1},k+1}} - X_{d_{j_{k},k}}\big).
\end{align*}
We notice that either $i_{n} = j_{n}$ or instead $n\geq 1$, $j_{n}\geq 2$ and $i_{n} = j_{n}-1$, since $0 < t-s$ $< 2^{-n}(T-r)$. In both cases, we have $(d_{i_{n},n},d_{j_{n},n})\in\Delta_{n}$. Moreover, $(d_{i_{k},k},d_{i_{k+1},k+1}), (d_{j_{k},k},d_{j_{k+1},k+1})\in\Delta_{k+1}$ for all $k\in\mathbb{N}_{0}$, by construction. So,
\[
|X_{s} - X_{t}| \leq 2\sum_{k=n}^{\infty} Y_{k}.
\]
Clearly, for each $s,t\in\mathbb{D}$ with $0 < t-s < T-r$ there is a unique $n\in\mathbb{N}_{0}$ satisfying $2^{-n-1}(T-r)\leq t-s < 2^{-n}(T-r)$. This entails that
\begin{equation}\label{Hoelder Moment Condition II}
\sup_{s,t\in [r,T]:\, s\neq t} \frac{|X_{s} - X_{t}|}{|s-t|^{\alpha}}\leq 2^{1+\alpha}(T-r)^{-\alpha}\sum_{k=0}^{\infty}2^{\alpha k}Y_{k},
\end{equation}
as $\mathbb{D}$ is a countable dense set in $[r,T]$ containing $T$ and $X$ is right-continuous. Hence,~\eqref{Hoelder Moment Condition II}, the triangle inequality, monotone convergence and~\eqref{Hoelder Moment Condition I} yield that
\[
\bigg(E\bigg[\sup_{s,t\in [r,T]:\,s\neq t}\frac{|X_{s}- X_{t}|^{p}}{|s-t|^{\alpha p}}\bigg]\bigg)^{1/p}
\leq 2^{1+\alpha} c_{0}^{1/p}(T-r)^{(1+q)/p - \alpha}\sum_{k=0}^{\infty}2^{(\alpha - q/p)k}.
\]
Since the power series on the right-hand side converges absolutely to the inverse of $1 - 2^{\alpha - q/p}$, the proposition follows.
\end{proof}

\subsection{Convergence along a sequence of partitions}\label{Convergence along a sequence of partitions}

We state a sufficient criterion for a sequence of processes to converge in the norm $\|\cdot\|_{\alpha,r}$ in probability, where $\alpha\in [0,1]$. For this purpose, condition~\eqref{Partition Condition} on the sequence of partitions is crucial.

\begin{lemma}\label{Hoelder Convergence Lemma}
Let $(\null_{n}X)_{n\in\mathbb{N}}$ be a sequence of $\mathbb{R}^{m}$-valued right-continuous processes for which there are $p,q > 0$ with $q \leq p$ such that for each $\beta\in (0,q/p)$ there is $c_{\beta}\geq 0$ satisfying
\begin{equation}\label{Hoelder Convergence Condition}
P\bigg(\max_{j\in\{0,\dots,k_{n}-1\}} \sup_{s,t\in [t_{j,n},t_{j+1,n}]:\,s\neq t} \frac{|\null_{n}X_{s} - \null_{n}X_{t}|}{|s- t|^{\beta}} \geq \lambda\bigg) \leq c_{\beta} \lambda^{-p}
\end{equation}
for every $n\in\mathbb{N}$ and $\lambda > 0$. If $(\|\null_{n}X^{r}\|_{\infty})_{n\in\mathbb{N}}$ and $(\max_{j\in\{1\dots,k_{n}\}}|\null_{n}X_{t_{j,n}}|/|\mathbb{T}_{n}|^{\alpha})_{n\in\mathbb{N}}$ converge in probability to zero, then so does $(\|\null_{n}X\|_{\alpha,r})_{n\in\mathbb{N}}$ for any $\alpha\in [0,q/p)$.
\end{lemma}

\begin{proof}
Let $\beta\in (\alpha,q/p)$ and $n\in\mathbb{N}$. First, a case distinction shows that
\begin{align*}
\sup_{s,t\in [r,T]:\,s\neq t}\frac{|\null_{n}X_{s} - \null_{n}X_{t}|}{|s-t|^{\alpha}} &\leq 2\max_{j\in\{0,\dots,k_{n}-1\}}\sup_{s,t\in [t_{j,n},t_{j+1,n}]:\,s\neq t}\frac{|\null_{n}X_{s} - \null_{n}X_{t}|}{|s-t|^{\alpha}}\\
&\quad + \max_{i,j\in\{1,\dots,k_{n}\}:\,i\neq j}\frac{|\null_{n}X_{t_{i,n}} - \null_{n}X_{t_{j,n}}|}{|t_{i,n} - t_{j,n}|^{\alpha}}.
\end{align*}
By using the facts that $|s-t|^{\beta -\alpha}\leq |\mathbb{T}_{n}|^{\beta -\alpha}$ and $|t_{i,n}- t_{j,n}|\geq |\mathbb{T}_{n}|/c_{\mathbb{T}}$ for all $i,j\in\{0,\dots,k_{n}\}$ with $i\neq j$ and $s,t\in [t_{j,n},t_{j+1,n}]$, we see that
\begin{align*}
P\bigg(\sup_{s,t\in [r,T]:\,s\neq t} \frac{|\null_{n}X_{s} - \null_{n}X_{t}|}{|s-t|^{\alpha}}\geq \varepsilon\bigg) &\leq c_{\beta}(4/\varepsilon)^{p}|\mathbb{T}_{n}|^{(\beta - \alpha)p}\\
&\quad + P\bigg(\max_{j\in\{1,\dots,k_{n}\}}|\null_{n}X_{t_{j,n}}|/|\mathbb{T}_{n}|^{\alpha} > (\varepsilon/4)c_{\mathbb{T}}^{-\alpha}\bigg)
\end{align*}
for any $\varepsilon > 0$. As the terms on the right-hand side converge to zero as $n\uparrow\infty$, the assertion is shown.
\end{proof}

\begin{remark}\label{Hoelder Convergence Remark}
Let $p\geq 1$ and $c_{0}\geq 0$ be such that $E[|\null_{n}X_{s}-\null_{n}X_{t}|^{p}]$ $\leq c_{0}|s-t|^{1+q}$ for all $n\in\mathbb{N}$, $j\in\{0,\dots,k_{n}-1\}$ and $s,t\in [t_{j,n},t_{j+1,n}]$. Then Chebyshev's inequality in combination with Proposition~\ref{Kolmogorov-Chentsov Proposition} ensure that condition~\eqref{Hoelder Convergence Condition} is satisfied.
\end{remark}

\subsection{Adapted linear interpolations of Brownian motion}\label{Adapted linear interpolations of Brownian motion}

We study the sequence $(\null_{n}W)_{n\in\mathbb{N}}$ of adapted linear interpolations of $W$ that are given by~\eqref{Adapted Linear Interpolation of Brownian Motion Definition} and whose paths lie in $H_{r}^{1}([0,T],\mathbb{R}^d)$. To this end, we introduce the following notation. For given $n\in\mathbb{N}$ and $t\in [r,T)$, let $i\in\{0,\dots,k_{n}-1\}$ be such that $t\in [t_{i,n},t_{i+1,n})$, then we set
\begin{equation*}
\underline{t}_{n}:=t_{(i-1)\vee 0,n},\quad t_{n}:=t_{i,n}\quad\text{and}\quad \overline{t}_{n}:=t_{i+1,n}.
\end{equation*}
That is, $\underline{t}_{n}$ is the predecessor of $t_{n}$ with respect to $\mathbb{T}_{n}$, unless $i=0$, and $\overline{t}_{n}$ is the successor of $t_{n}$. We also set $\underline{T}_{n}:=t_{k_{n-1,n}}$, $T_{n}:=T$ and $\overline{T}_{n}:=T$. In addition, we use the following abbreviations:
\begin{equation*}
\Delta t_{i,n}:= t_{i,n} - t_{(i-1)\vee 0,n}\quad\text{and}\quad \Delta W_{t_{i,n}}:= W_{t_{i,n}} - W_{t_{(i-1)\vee 0,n}}
\end{equation*}
for each $i\in\{0,\dots,k_{n}\}$. After these preparations, let us begin with a general integral representation.

\begin{lemma}\label{Brownian Linear Interpolation Property Lemma}
Let $n\in\mathbb{N}$ and $s,t\in [r,T]$ be such that $s < t$. Then each $\mathbb{R}^{m\times d}$-valued progressively measurable process $X$ satisfies
\[
\int_{s}^{t}X_{\underline{u}_{n}}\,d\null_{n}W_{u} = \frac{t-s}{\Delta t_{i+1,n}}\int_{t_{i-1,n}}^{t_{i,n}}X_{u_{n}}\,dW_{u}\quad\text{a.s.,}
\]
whenever $i\in\{1,\dots,k_{n}-1\}$ is such that $s,t\in [t_{i,n},t_{i+1,n}]$, and
\begin{align*}
\int_{s}^{t}X_{\underline{u}_{n}}\,d\null_{n}W_{u} &= \frac{t_{i+1,n} - s}{\Delta t_{i+1,n}}\int_{t_{i-1,n}}^{t_{i,n}}X_{u_{n}}\,dW_{u} + \int_{t_{i,n}}^{t_{j-1,n}}X_{u_{n}}\,dW_{u}\\
&\quad + \frac{t-t_{j,n}}{\Delta t_{j+1,n}}\int_{t_{j-1,n}}^{t_{j,n}}X_{u_{n}}\,dW_{u}\quad\text{a.s.,}
\end{align*}
if $i,j\in\{1,\dots,k_{n}-1\}$ are such that $i< j$, $s\in [t_{i,n},t_{i+1,n}]$ and $t\in [t_{j,n},t_{j+1,n}]$.
\end{lemma}

\begin{proof}
The first identity follows immediately from the definition of $\null_{n}W$. To obtain the second, we simply use the decomposition
\begin{align*}
\int_{s}^{t} X_{\underline{u}_{n}}\,d\null_{n}W_{u} &= \int_{s}^{t_{i+1,n}}X_{\underline{u}_{n}}\,d\null_{n}W_{u} + \sum_{k=i+1}^{j-1} \int_{t_{k,n}}^{t_{k+1,n}}X_{\underline{u}_{n}}\,d\null_{n}W_{u} + \int_{t_{j,n}}^{t}X_{\underline{u}_{n}}\,d\null_{n}W_{u}
\end{align*}
together with the first identity.
\end{proof}

Let us recall an explicit moment estimate for stochastic integrals driven by $W$ from~\cite{MaoSDEandApp}[Theorem 7.2]. For $p\geq 2$ we set $ w_{p}:=((p^{3}/2)/(p-1))^{p/2}$, then for any $\mathbb{R}^{m\times d}$-valued progressively measurable process $X$ with $\int_{r}^{T}E[|X_{u}|^{p}]\,du < \infty$,
\begin{equation}\label{Mao's Inequality}
E\bigg[\sup_{v\in [s,t]}\bigg|\int_{s}^{v}X_{u}\,dW_{u}\bigg|^{p}\bigg] \leq w_{p} (t-s)^{p/2-1}\int_{s}^{t}E[|X_{u}|^{p}]\,du
\end{equation}
for all $s,t\in [r,T]$ with $s\leq t$. We derive a corresponding result for the sequence $(\null_{n}W)_{n\in\mathbb{N}}$ of adapted linear interpolations of $W$.

\begin{proposition}\label{Adapted Linear Interpolation Estimation}
For each $p\geq 2$ there is $\hat{w}_{p} > 0$ such that each $\mathbb{R}^{m\times d}$-valued progressively measurable process $X$ satisfies
\begin{equation*}
E\bigg[\max_{v\in [s,t]}\bigg|\int_{s}^{v}X_{\underline{u}_{n}}\,d\null_{n}W_{u}\bigg|^{p}\bigg] \leq \hat{w}_{p}(t-s)^{p/2}\max_{j\in\{0,\dots,k_{n}\}:\,t_{j,n}\in [\underline{s}_{n},\underline{t}_{n}]} E[|X_{t_{j,n}}|^{p}] 
\end{equation*}
for each $n\in\mathbb{N}$ and $s,t\in [r,T]$ with $s\leq t$.
\end{proposition}

\begin{proof}
We assume that $E[|X_{t_{j,n}}|^{p}] < \infty$ for all $j\in\{0,\dots,k_{n}\}$ with $t_{j,n}\in [\underline{s}_{n},\underline{t}_{n}]$, as otherwise there is nothing to show. First, if $t\leq t_{1,n}$, then $\int_{s}^{v}X_{\underline{u}_{n}}\,d\null_{n}W_{u} = 0$ for each $v\in [s,t]$. For $s < t_{1,n}$ and $t\geq t_{1,n}$ we have
\[
\int_{s}^{v}X_{\underline{u}_{n}}\,d\null_{n}W_{u} = \int_{t_{1,n}}^{v}X_{\underline{u}_{n}}\,d\null_{n}W_{u}\quad\text{for all $v\in [t_{1,n},t]$.}
\]
Thus, let us use Lemma~\ref{Brownian Linear Interpolation Property Lemma} and assume at first that $s,t\in [t_{i,n},t_{i+1,n}]$ for some $i\in\{1,\dots,k_{n}-1\}$. Then~\eqref{Mao's Inequality} yields that
\[
E\bigg[\max_{v\in [s,t]}\bigg|\int_{s}^{v}X_{\underline{u}_{n}}\,d\null_{n}W_{u}\bigg|^{p}\bigg]\leq w_{p}c_{\mathbb{T}}^{p/2} (t-s)^{p/2} E\big[|X_{t_{i-1,n}}|^{p}\big],
\]
where $c_{\mathbb{T}}$ is the constant appearing in~\eqref{Partition Condition}. Let now $i,j\in\{1,\dots,k_{n}-1\}$ be such that $i < j$, $s\in [t_{i,n},t_{i+1,n}]$ and $t\in [t_{j,n},t_{j+1,n}]$, then
\begin{align*}
&\max_{v\in[s,t]}\bigg|\int_{s}^{v}X_{\underline{u}_{n}}\,d\null_{n}W_{u}\bigg| \leq \frac{t_{i+1,n} - s}{\Delta t_{i+1,n}}\bigg|\int_{t_{i-1,n}}^{t_{i,n}}X_{u_{n}}\,dW_{u}\bigg|\\
&\quad + \max_{k\in\{i,\dots,j-1\}}\bigg|\int_{t_{i,n}}^{t_{k,n}}X_{u_{n}}\,dW_{u}\bigg| + \frac{t - t_{j,n}}{\Delta t_{j+1,n}}\bigg|\int_{t_{j-1,n}}^{t_{j,n}}X_{u_{n}}\,dW_{u}\bigg|\quad\text{a.s.}
\end{align*}
This is due to Lemma~\ref{Brownian Linear Interpolation Property Lemma}, which asserts that the adapted process $[s,t]\times\Omega\rightarrow\mathbb{R}^{m}$, $(v,\omega)\mapsto \int_{s}^{v}X_{\underline{u}_{n}}(\omega)\,d\null_{n}W_{u}(\omega)$ is piecewise linear. Hence, from~\eqref{Mao's Inequality} we obtain that
\[
E\bigg[\max_{v\in [s,t]}\bigg|\int_{s}^{v}X_{\underline{u}_{n}}\,d\null_{n}W_{u}\bigg|^{p}\bigg] \leq  \hat{w}_{p}(t-s)^{p/2}\max_{k\in\{i-1,\dots,j-1\}} E\big[|X_{t_{k,n}}|^{p}\big] 
\]
for $\hat{w}_{p}:=3^{p}w_{p}c_{\mathbb{T}}^{p/2}$, which yields the claim.
\end{proof}

Finally, we derive an explicit integral moment estimate for $(\null_{n}W)_{n\in\mathbb{N}}$.

\begin{lemma}\label{Brownian Interpolation Integral Moments Lemma}
For each $p,q\geq 1$ there exists $\hat{w}_{p,q} > 0$ satisfying
\begin{equation}\label{Brownian Interpolation Constant}
E\bigg[\bigg(\int_{s}^{t}|_{n}\dot{W}_{u}|^{q}\,du\bigg)^{p}\bigg] \leq \hat{w}_{p,q}|\mathbb{T}_{n}|^{-pq/2} (t-s)^{p}
\end{equation}
for all $n\in\mathbb{N}$ and $s,t\in [r,T]$ with $s\leq t$.
\end{lemma}

\begin{proof}
Since $\null_{n}W$ is constant on $[r,t_{1,n}]$, we let at first $s,t\in [t_{i,n},t_{i+1,n}]$ for some $i\in\{1,\dots,k_{n}-1\}$ and $Z$ be an $\mathbb{R}^{d}$-valued random vector such that $Z\sim\mathcal{N}(0,\mathbbm{I}_{d})$. Then
\[
E\bigg[\bigg(\int_{s}^{t}|\null_{n}\dot{W}_{u}|^{q}\,du\bigg)^{p}\bigg]= E\big[|Z|^{pq}\big]\frac{(\Delta t_{i,n})^{pq/2}}{(\Delta t_{i+1,n})^{pq}}(t-s)^{p}\leq \hat{w}_{p,q}|\mathbb{T}_{n}|^{-pq/2}(t-s)^{p}
\]
for $\hat{w}_{p,q}:=E[|Z|^{pq}]c_{\mathbb{T}}^{pq}$, where $c_{\mathbb{T}}$ is the constant in~\eqref{Partition Condition}. Next, assume instead $i,j\in\{1,\dots,k_{n}-1\}$ are such that $i < j$, $s\in [t_{i,n},t_{i+1,n}]$ and $t\in [t_{j,n},t_{j+1,n}]$. In this case,
\[
\bigg(E\bigg[\bigg(\int_{s}^{t}|\null_{n}\dot{W}_{u}|^{q}\,du\bigg)^{p}\bigg]\bigg)^{1/p}\leq \hat{w}_{p,q}^{1/p}|\mathbb{T}_{n}|^{-q/2}(t-s),
\]
by the triangle inequality. Therefore, the assertion holds.
\end{proof}

\subsection{Auxiliary convergence results}\label{Auxiliary convergence results}

In this section we provide interpolation error estimates in the supremum norm for stochastic processes and several moment estimates, required to prove~\eqref{General SDE Partition Limit Equation}.

\begin{lemma}\label{Delayed Linear Interpolation Lemma}
Let $n\in\mathbb{N}$ and $x:[0,T]\rightarrow\mathbb{R}^{m}$. Then the map $L_{n}(x):[0,T]\rightarrow\mathbb{R}^{m}$ given at~\eqref{Delayed Linear Interpolation Operator} satisfies $\|L_{n}(x)^{t}\|_{\infty}$ $\leq\|x^{r}\|_{\infty}\vee\max_{j\in\{1,\dots,k_{n}-1\}:\,t_{j,n} < t}|x(t_{j,n})|$ and
\[
\|L_{n}(x)^{t} - x^{t}\|_{\infty}\leq \max_{j\in\{0,\dots,k_{n}-1\}:\,t_{j,n}\leq t}\sup_{s\in [t_{j,n},t_{j+1,n}]}|x(t_{(j-1)\vee 0,n})-x^{t}(s)|\vee |x(t_{j,n})-x^{t}(s)|
\]
for each $t\in [t_{1,n},T]$.
\end{lemma}

\begin{proof}
Fix $s\in [t_{1,n},t]$ and let $i\in\{1,\dots,k_{n}-1\}$ be such that $s\in [t_{i,n},t_{i+1,n}]$, then $|L_{n}(x)(s)|$ $\leq |x(t_{i-1,n})|\vee |x(t_{i,n})|$, since $L_{n}(x)$ is linear on $[t_{i,n},t_{i+1,n}]$. In addition,
\begin{align*}
|L_{n}(x)(s) - x(s)|&\leq \frac{t_{i+1,n} - s}{\Delta t_{i+1,n}}|x(t_{i-1,n}) - x(s)| + \frac{s - t_{i,n}}{\Delta t_{i+1,n}}|x(t_{i,n}) - x(s)|\\
&\leq \sup_{u\in [t_{i,n},t_{i+1,n}]}|x(t_{i-1,n})-x^{t}(u)|\vee |x(t_{i,n})-x^{t}(u)|,
\end{align*}
which is readily seen, and the assertions follow.
\end{proof}

In combination with Proposition~\ref{Kolmogorov-Chentsov Proposition}, this directly gives the following result.

\begin{lemma}\label{Auxiliary Convergence Result 2}
Let $(\null_{n}X)_{n\in\mathbb{N}}$ be a sequence of $\mathbb{R}^{m}$-valued right-continuous processes for which there are $c_{0}\geq 0$, $p\geq 1$ and $q > 0$ such that
\[
E\big[|\null_{n}X_{s} - \null_{n}X_{t}|^{p}\big]\leq c_{0}|s-t|^{1 + q}
\]
for all $n\in\mathbb{N}$, $j\in\{0,\dots,k_{n}-1\}$ and $s,t\in [t_{j,n},t_{j+1,n}]$. Then there is $c_{p,q} > 0$ such that
\[
E[\|L_{n}(\null_{n}X) - \null_{n}X\|_{\infty}^{p}] \leq c_{p,q}c_{0}|\mathbb{T}_{n}|^{q}\quad\text{for any $n\in\mathbb{N}$.}
\]
\end{lemma}

\begin{proof}
For $\alpha\in [0,1]$ Lemma~\ref{Delayed Linear Interpolation Lemma} entails that
\begin{align*}
\|L_{n}(\null_{n}X)-\null_{n}X\|_{\infty}&\leq |\mathbb{T}_{n}|^{\alpha}\max_{j\in\{0,\dots,k_{n}-1\}}\sup_{s,t\in [t_{j,n},t_{j+1,n}]:\,s\neq t} \frac{|\null_{n}X_{s} - \null_{n}X_{t}|}{|s-t|^{\alpha}}\\
&\quad + \max_{j\in\{1,\dots,k_{n}-1\}} |\null_{n}X_{t_{j,n}} - \null_{n}X_{t_{j-1,n}}|.
\end{align*}
Let additionally $\alpha < q/p$. Since we have $\sum_{j=0}^{k_{n-1}} (t_{j+1,n} -t_{j,n}) = T-r$, it follows by virtue of Proposition~\ref{Kolmogorov-Chentsov Proposition} that
\begin{equation*}
E\bigg[\max_{j\in\{0,\dots,k_{n}-1\}}\sup_{s,t\in [t_{j,n},t_{j+1,n}]:\,s\neq t} \frac{|\null_{n}X_{s} - \null_{n}X_{t}|^{p}}{|s-t|^{\alpha p}}\bigg] \leq k_{\alpha,p,q}c_{0}(T-r)|\mathbb{T}_{n}|^{q - \alpha p},
\end{equation*}
where the constant $k_{\alpha,p,q}$ is given by~\eqref{Kolmogorov-Chentsov Constant}. Moreover,
\[
E\big[ \max_{j\in\{1,\dots,k_{n}-1\}} |\null_{n}X_{t_{j,n}} - \null_{n}X_{t_{j-1,n}}|^{p}\big] \leq c_{0}(T-r)|\mathbb{T}_{n}|^{q}.
\]
As the function $[0,q/p)\rightarrow (2^{p},\infty)$, $\beta\mapsto k_{\beta,p,q}$ is strictly increasing, we choose $\alpha=0$ and set $c_{p,q}:=2^{p-1}(1 + k_{0,p,q})(T-r)$, which completes the proof.
\end{proof}

A consequence of Lemma~\ref{Brownian Interpolation Integral Moments Lemma} is the following moment bound.

\begin{lemma}\label{Auxiliary Convergence Result 1}
Let $(\null_{n}X)_{n\in\mathbb{N}}$ be a sequence of $\mathbb{R}_{+}$-valued measurable processes for which there are $p > 2$ and $c_{p} > 0$ such that $E[\null_{n}X_{s}^{p}] \leq c_{p}|\mathbb{T}_{n}|^{p}$ for each $s\in [r,T)$ and $n\in\mathbb{N}$. Then there is $c_{2} > 0$ satisfying
\[
E\bigg[\bigg(\int_{r}^{T}\null_{n}X_{s}|_{n}\dot{W}_{s}|\,ds\bigg)^{2}\bigg]\leq c_{2}|\mathbb{T}_{n}|\quad\text{for all $n\in\mathbb{N}$.}
\]
\end{lemma}

\begin{proof}
Let $q > 1$ satisfy $2/p + 1/q = 1$, then it follows from the inequalities of Cauchy-Schwarz and \holder that
\begin{align*}
E\bigg[\bigg(\int_{r}^{T}\null_{n}X_{s} |\null_{n}\dot{W}_{s}|\,ds\bigg)^{2}\bigg]\leq \bigg(E\bigg[\bigg(\int_{r}^{T}\null_{n}X_{s}^{2}\,ds\bigg)^{p/2}\bigg]\bigg)^{2/p} c_{2,1}|\mathbb{T}_{n}|^{-1} \leq c_{2}|\mathbb{T}_{n}|,
\end{align*}
where we have set $c_{2,1}:= \hat{w}_{q,2}^{1/q}(T-r)$ and $c_{2}:=c_{p}^{2/p}(T-r)c_{2,1}$, by using the constant $\hat{w}_{q,2}$ constructed in Lemma~\ref{Brownian Interpolation Integral Moments Lemma}.
\end{proof}

To shorten the notation for the next and several other estimates in Section~\ref{Proof of the main result}, we introduce for each $n\in\mathbb{N}$ the function $\gamma_{n}:[r,T]\rightarrow [0,c_{\mathbb{T}}]$ defined via
\begin{equation}\label{Gamma Function}
\gamma_{n}(s) := \frac{\Delta s_{n}}{\Delta \overline{s}_{n}}.
\end{equation}
So, $\gamma_{n}$ vanishes on $[r,t_{1,n})$ and agrees with the constant $\Delta t_{i,n}/\Delta t_{i+1,n}$ on $[t_{i,n},t_{i+1,n})$ for each $i\in\{1,\dots,k_{n}-1\}$ and we have $\gamma_{n}(T) = 1$.

\begin{lemma}\label{Auxiliary Convergence Result 3}
Let $F:[r,T]\times C([0,T],\mathbb{R}^{m})\rightarrow\mathbb{R}^{m}$ be $d_{\infty}$-Lipschitz continuous and $(\null_{n}Y)_{n\in\mathbb{N}}$ be a sequence in $\mathscr{C}([0,T],\mathbb{R}^{m})$ for which there are $c_{0},c_{2,0}\geq 0$ such that $|F(t,x)| \leq c_{0}(1 + \|x\|_{\infty})$ and
\[
E\big[\|\null_{n}Y\|_{\infty}^{2}\big] + E\big[\|\null_{n}Y^{s} - \null_{n}Y^{t}\|_{\infty}^{2}\big]/|s-t|\leq c_{2,0}\big(1 + E\big[\|\null_{n}Y^{r}\|_{\infty}^{2}\big]\big)
\]
for all $n\in\mathbb{N}$, $s,t\in [r,T]$ with $s\neq t$ and $x\in C([0,T],\mathbb{R}^{m})$. Then there is $c_{2} > 0$ satisfying
\begin{equation*}
E\bigg[\max_{j\in\{0,\dots,k_{n}\}}\bigg|\int_{r}^{t_{j,n}}F(\underline{s}_{n},\null_{n}Y)(\gamma_{n}(s) - 1)\,ds\bigg|^{2}\bigg]\leq c_{2}|\mathbb{T}_{n}|\big(1 + E\big[\|\null_{n}Y^{r}\|_{\infty}^{2}\big]\big)
\end{equation*}
for each $n\in\mathbb{N}$.
\end{lemma}

\begin{proof}
We may assume that $E[\|\null_{n}Y^{r}\|_{\infty}^{2}] < \infty$ and, by decomposing the integral, we can rewrite that
\[
\int_{r}^{t_{j,n}}F(\underline{s}_{n},\null_{n}Y)\gamma_{n}(s)\,ds = \int_{r}^{t_{j-1,n}}F(s_{n},\null_{n}Y)\,ds
\]
for each $j\in\{1,\dots,k_{n}\}$. Thus, let $\lambda_{0}\geq 0$ be a Lipschitz constant for $F$, then
\[
E\bigg[\max_{j\in\{1,\dots,k_{n}\}}\bigg|\int_{r}^{t_{j-1,n}}F(s_{n},\null_{n}Y) - F(\underline{s}_{n},\null_{n}Y)\,ds\bigg|^{2}\bigg]\leq c_{2,1}|\mathbb{T}_{n}|\big(1 + E\big[\|\null_{n}Y^{r}\|_{\infty}^{2}\big]\big)
\]
with $c_{2,1}:=2(T-r)^{2}\lambda_{0}^{2}(1 + c_{2,0})$. In addition, we estimate that
\[
E\bigg[\max_{j\in\{1,\dots,k_{n}\}}\bigg|\int_{t_{j-1,n}}^{t_{j,n}}F(\underline{s}_{n},\null_{n}Y)\,ds\bigg|^{2}\bigg]\leq c_{2,2}|\mathbb{T}_{n}|^{2}\big(1 + E\big[\|\null_{n}Y^{r}\|_{\infty}^{2}\big]\big),
\]
where $c_{2,2} := 2c_{0}^{2}(1 +c_{2,0})$. So, the constant $c_{2}:=2(c_{2,1} + (T-r)c_{2,2})$ yields the claim.
\end{proof}

The last moment estimate involves integrals with respect to $\null_{n}W$ and $W$, where $n\in\mathbb{N}$, and it extends Lemma 3.2 in~\cite{SupportThm}:
\begin{proposition}\label{Auxiliary Convergence Result 4}
Let $F:[r,T]\times C([0,T],\mathbb{R}^{m})\rightarrow\mathbb{R}^{m\times d}$ be $d_{\infty}$-Lipschitz continuous and $(\null_{n}Y)_{n\in\mathbb{N}}$ be a sequence in $\mathscr{C}([0,T],\mathbb{R}^{m})$. Suppose there are $c_{0}\geq 0$, $p \geq 2$ and $c_{p,0}\geq 0$ such that $|F(t,x)|\leq c_{0}(1 + \|x\|_{\infty})$ and
\[
E[\|\null_{n}Y\|_{\infty}^{p}] + E[\|\null_{n}Y^{s} - \null_{n}Y^{t}\|_{\infty}^{p}\big]/|s-t|^{p/2}\leq c_{p,0}\big(1 + E[\|\null_{n}Y^{r}\|_{\infty}^{p}]\big)
\]
for each $n\in\mathbb{N}$, $s,t\in [r,T]$ with $s\neq t$ and $x\in C([0,T],\mathbb{R}^{m})$. Then there is $c_{p} > 0$ such that
\[
E\bigg[\max_{j\in\{0,\dots,k_{n}\}}\bigg|\int_{r}^{t_{j,n}}F(\underline{s}_{n},\null_{n}Y)\,d(\null_{n}W_{s} - W_{s})\bigg|^{p}\bigg] \leq c_{p}|\mathbb{T}_{n}|^{p/2-1}\big(1 + E[\|\null_{n}Y^{r}\|_{\infty}^{p}]\big)
\]
for any $n\in\mathbb{N}$.
\end{proposition}

\begin{proof}
We let $E[\|\null_{n}Y^{r}\|_{\infty}^{p}] < \infty$ and apply Lemma~\ref{Brownian Linear Interpolation Property Lemma} to obtain
\[
\int_{r}^{t_{j,n}}F(\underline{s}_{n},\null_{n}Y)\,d\null_{n}W_{s} = \int_{r}^{t_{j-1,n}}F(s_{n},\null_{n}Y)\,dW_{s}\quad\text{a.s.}
\]
for all $j\in\{1,\dots,k_{n}\}$. Let $\lambda_{0}\geq 0$ denote a Lipschitz constant for $F$, then
\begin{align*}
E\bigg[\max_{j\in\{1,\dots,k_{n}\}}&\bigg|\int_{r}^{t_{j-1,n}} F(s_{n},\null_{n}Y) - F(\underline{s}_{n},\null_{n}Y)\,dW_{s}\bigg|^{p}\bigg]\leq c_{p,1}|\mathbb{T}_{n}|^{p/2}\big(1 + E[\|\null_{n}Y^{r}\|_{\infty}^{p}]\big)
\end{align*}
with $c_{p,1}:=2^{p-1}w_{p}(T-r)^{p/2}\lambda_{0}^{p}(1 + c_{p,0})$, where $w_{p}$ satisfies~\eqref{Mao's Inequality}. Moreover,
\begin{align*}
E\bigg[\max_{j\in\{1,\dots,k_{n}\}}\bigg|\int_{t_{j-1,n}}^{t_{j,n}}F(\underline{s}_{n},\null_{n}Y)\,dW_{s}\bigg|^{p}\bigg]&\leq \sum_{j=1}^{k_{n}}E\bigg[\bigg|\int_{t_{j-1,n}}^{t_{j,n}}F(\underline{s}_{n},\null_{n}Y)\,dW_{s}\bigg|^{p}\bigg] 
\\
&\leq c_{p,2}|\mathbb{T}_{n}|^{p/2-1}\big(1 + E\big[\|\null_{n}Y^{r}\|_{\infty}^{p}]\big)
\end{align*}
for $c_{p,2}:=2^{p-1}w_{p}(T-r)c_{0}^{p}(1 + c_{p,0})$. So, we set $c_{p}:=2^{p-1}((T-r)c_{p,1} + c_{p,2})$ and obtain the asserted estimate.
\end{proof}

\section{Path-dependent ODEs and SDEs: proofs}\label{sec.proofs}

We give existence and uniqueness proofs for mild solutions to path-dependent ODEs in Section~\ref{Mild solutions to path-dependent ODEs} and strong solutions to path-dependent SDEs in Section~\ref{Strong solutions to path-dependent SDEs}.

\subsection{Proof of Proposition~\ref{General ODE Proposition}}\label{Proof of Proposition 3}

Let us first derive a global estimate for mild solutions to the ODE~\eqref{General ODE}. 

\begin{lemma}\label{General ODE Lemma 1}
Under~\eqref{C.1}, any mild solution $x$ to~\eqref{General ODE} satisfies
\begin{equation}\label{General ODE Inequality 1}
\|x^{t}\|_{H,r}^{2} \leq c_{H}e^{c_{H}\int_{r}^{t}c_{0}(s)^{2}\,ds}\bigg(\|x^{r}\|_{\infty}^{2} + \int_{r}^{t}c_{0}(s)^{2}\,ds\bigg)
\end{equation}
for each $t\in [r,T]$ with $c_{H}:=2^{2}\max\{1,T-r\}$.
\end{lemma}

\begin{proof}
We readily estimate that
\[
\|x^{t}\|_{H,r}^{2} \leq 2\|x^{r}\|_{\infty}^{2} + 2\int_{r}^{t}c_{0}(s)^{2}\bigg(1 + \|x^{r}\|_{\infty} + \int_{r}^{s}|\dot{x}(u)|\,du\bigg)^{2}\,ds
\]
for all $t\in [r,T]$, which shows that $\|x\|_{H,r}$ is finite. Hence, the claim follows from Gronwall's inequality.
\end{proof}

Now we check the uniqueness of mild solutions, which implies uniqueness for classical solutions.

\begin{lemma}\label{General ODE Lemma 2}
Assume that~\eqref{C.1} and~\eqref{C.2} hold. Then any two mild solutions $x$ and $\tilde{x}$ to~\eqref{General ODE} that satisfy $x^{r} = \tilde{x}^{r}$ must coincide. 
\end{lemma}

\begin{proof}
By Lemma~\ref{General ODE Lemma 1}, there is $n\in\mathbb{N}$ such that $\|x\|_{H,r}\vee \|\tilde{x}\|_{H,r}\leq n$. Thus,
\begin{align*}
\|x^{t} - \tilde{x}^{t}\|_{H,r}^{2} &\leq \int_{r}^{t}\lambda_{n}^{2}(s)\|x^{s} - \tilde{x}^{s}\|_{H,r}^{2}\,ds
\end{align*}
for every $t\in [r,T]$ and Gronwall's inequality implies that $x=\tilde{x}$.
\end{proof}

\begin{proof}[Proof of Proposition~\ref{General ODE Proposition}]
Since uniqueness follows from Lemma~\ref{General ODE Lemma 2}, we directly turn to the existence assertion and define $\mathscr{H}$ to be the closed and bounded set of all $x\in H_{r}^{1}([0,T],\mathbb{R}^{m})$ satisfying $x^{r} =\hat{x}^{r}$ and the estimate~\eqref{General ODE Inequality 1}.

According to Lemma~\ref{General ODE Lemma 1}, a path $x\in C([0,T],\mathbb{R}^{m})$ is a mild solution to~\eqref{General ODE} such that $x^{r} = \hat{x}^{r}$ if and only if $x\in\mathscr{H}$ and it is a fixed-point of the operator $\psi:\mathscr{H}\rightarrow H_{r}^{1}([0,T],\mathbb{R}^{m})$ given by
\[
\psi(y)(t):= x_{0}(t) + \int_{r}^{r\vee t}F(s,y)\,ds.
\]
We remark that $\psi$ maps $\mathscr{H}$ into itself. Indeed, this follows by inserting~\eqref{General ODE Inequality 1} into the inequality $\|\psi(x)^{t}\|_{H,r}^{2}\leq c_{H}\|x_{0}\|_{\infty}^{2} + c_{H}\int_{r}^{t}c_{0}(s)^{2}(1 + \|x^{s}\|_{H,r}^{2})\,ds$, valid for every $x\in\mathscr{H}$ and $t\in [r,T]$. Because $x_{n} = \psi(x_{n-1})$ for each $n\in\mathbb{N}$, we now know that $(x_{n})_{n\in\mathbb{N}_{0}}$ is a sequence in $\mathscr{H}$.

Next, let us choose $l\in\mathbb{N}$ such that $\|x\|_{H,r}\leq l$ for all $x\in\mathscr{H}$. Then we get that $\|\psi(x)^{t} - \psi(\tilde{x})^{t}\|_{H,r}^{2}\leq \int_{r}^{t}\lambda_{l}(s)^{2}\|x^{s} - \tilde{x}^{s}\|_{H,r}^{2}\,ds$ for any $x,\tilde{x}\in\mathscr{H}$ and $t\in [r,T]$, which shows that $\psi$ is $\|\cdot\|_{H,r}$-Lipschitz continuous. It follows inductively that
\[
\|x_{n+1}^{t} - x_{n}^{t}\|_{H,r}^{2}\leq \frac{\delta^{2}}{n!}\bigg(\int_{r}^{t}\lambda_{l}(s)^{2}\,ds\bigg)^{n}
\]
for every $n\in\mathbb{N}_{0}$ and $t\in [r,T]$, where we have set $\delta:=\|\psi(x_{0}) - x_{0}\|_{H,r}$. Hence, the triangle inequality gives us that
\[
\|x_{n} - x_{k}\|_{H,r}\leq \delta\sum_{i=k}^{n-1}\bigg(\frac{1}{i!}\bigg)^{1/2}\bigg(\int_{r}^{T}\lambda_{l}(s)^{2}\,ds\bigg)^{i/2}
\]
for all $k,n\in\mathbb{N}_{0}$ with $k < n$. The ratio test yields that the series $\sum_{i=0}^{\infty}(1/i!)^{1/2}u^{i/2}$ converges absolutely for all $u\geq 0$. So, $\lim_{k\uparrow\infty}\sup_{n\in\mathbb{N}:\,n\geq k} \|x_{n} - x_{k}\|_{H,r} = 0$. 

Since $\mathscr{H}$ is closed with respect to the complete norm $\|\cdot\|_{H,r}$, there is a unique $x\in\mathscr{H}$ such that $\lim_{n\uparrow\infty} \|x_{n} - x\|_{H,r}=0$. Lipschitz continuity of $\psi$ implies that $\lim_{n\uparrow\infty} \|x_{n+1} - \psi(x)\|_{H,r}=0$. For this reason, $x = \psi(x)$ and the proposition is established.
\end{proof}

\subsection{Proof of Proposition~\ref{General SDE Proposition}}\label{sec.sdeproof}

At first, let us deduce a global estimate for any solution to the SDE~\eqref{SDE}. For this purpose, whenever $p > 4$ and~\eqref{C.3} holds, we set $
c_{p}:=(\int_{r}^{T}c_{0}(s)^{2}\,ds)^{p/2} + w_{p} \tilde{c}_{0}^{p}$, where $w_{p}$ is the constant appearing in~\eqref{Mao's Inequality}.

\begin{lemma}\label{General SDE Lemma 1}
Let~\eqref{C.3} be valid. Then for each $p > 4$ and $\alpha\in [0,1/2-2/p)$, any solution $X$ to~\eqref{SDE} satisfies
\begin{equation}\label{General SDE Inequality 1}
E[\|X^{t}\|_{\alpha,r}^{p}] \leq c_{\alpha,p}e^{c_{\alpha,p}c_{p}(t-r)}\big(E[\|X^{r}\|_{\infty}^{p}] + c_{p}(t-r)\big)
\end{equation}
for all $t\in [r,T]$ with $c_{\alpha,p}:=8^{p-1}\max\{1,T-r\}^{p/2-1}k_{\alpha,p,p/2-2}$, where $k_{\alpha,p,p/2-2}$ is given by~\eqref{Kolmogorov-Chentsov Constant} for the choice $q=p/2-2$.
\end{lemma}

\begin{proof}
We assume that $E[\|X^{r}\|_{\infty}^{p}] < \infty$ and let $n\in\mathbb{N}$. Then the stopping time $\tau_{n}:=\inf\{t\in [0,T]\,|\,|X_{t}|\geq n\}\vee r$ satisfies $\|X^{\tau_{n}}\|_{\infty}\leq \|X^{r}\|_{\infty}\vee n$ and we get that
\[
E[|X_{u}^{\tau_{n}} - X_{v}^{\tau_{n}}|^{p}] \leq 4^{p-1} c_{p}(v-u)^{p/2-1}\int_{r}^{t}1 + E[\|X^{s\wedge\tau_{n}}\|_{\infty}^{p}]\,ds
\]
for every $t\in [r,T]$ and $u,v\in [r,t]$ with $u \leq v$, by the inequalities of Jensen and Cauchy-Schwarz and~\eqref{Mao's Inequality}. Therefore, it follows from Proposition~\ref{Kolmogorov-Chentsov Proposition} that
\begin{align*}
E[\|X^{t\wedge\tau_{n}}\|_{\alpha,r}^{p}]&\leq 2^{p-1}E[\|X^{r}\|_{\infty}^{p}]\\
&\quad + 8^{p-1}k_{\alpha,p,p/2-2}c_{p}(t-r)^{p(1/2-\alpha)-1}\int_{r}^{t}1 + E[\|X^{s\wedge\tau_{n}}\|_{\infty}^{p}]\,ds,
\end{align*}
showing in particular that $E[\|X^{\tau_{n}}\|_{\alpha,r}^{p}]$ is finite. Thus, Gronwall's inequality and Fatou's lemma lead to the claimed estimate.
\end{proof}

\begin{remark}\label{General SDE Remark 2}
If instead $\alpha\in [0,1/2)$ and $p\geq 1$ satisfy $\alpha\geq 1/2 - 2/p$, then we still have that $E[\|X\|_{\alpha,r}^{p}]\leq (E[\|X\|_{\alpha,r}^{q}])^{p/q} < \infty$ for any $q > \max\{p,4\}$ with $\alpha < 1/2 - 2/q$, by \holders inequality.
\end{remark}

\begin{lemma}\label{General SDE Lemma 2}
Under~\eqref{C.5}, pathwise uniqueness holds for~\eqref{SDE}.
\end{lemma}

\begin{proof}
Let $X$ and $\tilde{X}$ be two weak solutions to~\eqref{SDE} defined on a common filtered probability space $(\tilde{\Omega},\tilde{\mathscr{F}},(\tilde{\mathscr{F}}_{t})_{t\in [0,T]},\tilde{P})$ that satisfies the usual conditions and on which there is a standard $d$-dimensional $(\tilde{\mathscr{F}}_{t})_{t\in [0,T]}$-Brownian motion $\tilde{W}$ such that $X^{r} = \tilde{X}^{r}$ a.s. 

For fixed $n\in\mathbb{N}$ we set $\tau_{n} :=\inf\{t\in [0,T]\,|\,|X_{t}|\geq n\text{ or } |\tilde{X}_{t}|\geq n\}\vee r$, ensuring that $\|X^{\tau_{n}}\|_{\infty}\vee \|\tilde{X}^{\tau_{n}}\|_{\infty}\leq \|X^{r}\|_{\infty}\vee n$ and $\|X^{\tau_{n}}-\tilde{X}^{\tau_{n}}\|_{\infty}\leq 2n$ a.s. Then with $c_{2}:=2(T-r + w_{2})$ we obtain that
\[
\tilde{E}\big[\|X^{t\wedge\tau_{n}} - \tilde{X}^{t\wedge\tau_{n}}\|_{\infty}^{2}\big]\leq c_{2}\int_{r}^{t}\lambda_{n}(s)^{2}\tilde{E}\big[\|X^{s\wedge\tau_{n}} - \tilde{X}^{s\wedge\tau_{n}}\|_{\infty}^{2}\big]\,ds
\]
for any $t\in [r,T]$. So, $X^{\tau_{n}} = \tilde{X}^{\tau_{n}}$ a.s., by Gronwall's inequality. As $\tau_{n}\leq \tau_{n+1}$ for all $n\in\mathbb{N}$ and $\sup_{n\in\mathbb{N}}\tau_{n} = \infty$, we get that $X_{t} = \lim_{n\uparrow\infty} X_{t}^{\tau_{n}}$ $= \lim_{n\uparrow\infty}\tilde{X}_{t}^{\tau_{n}} = \tilde{X}_{t}$ a.s.~for all $t\in [r,T]$. Right-continuity of paths implies that $X=\tilde{X}$ a.s.
\end{proof}

\begin{proof}[Proof of Proposition~\ref{General SDE Proposition}]
We define $\mathscr{H}$ to be the bounded set of all $X\in\mathscr{C}([0,T],\mathbb{R}^{m})$ satisfying $X^{r} = \hat{x}^{r}$ a.s.~and the estimate~\eqref{General SDE Inequality 1} for any $p > 4$ and $\alpha\in [0,1/2 -2/p)$. Then Remark~\ref{General SDE Remark 2} entails that $\mathscr{H}\subset\mathscr{C}_{r,\infty}^{1/2-}([0,T],\mathbb{R}^{m})$.

By Lemma~\ref{General SDE Lemma 1}, a process $X\in\mathscr{C}([0,T],\mathbb{R}^{m})$ is a solution to~\eqref{SDE} satisfying $X^{r} = \hat{x}^{r}$ a.s.~if and only if $X\in\mathscr{H}$ and it is an a.s.~fixed-point of the operator $\Psi:\mathscr{H}\rightarrow\mathscr{C}_{r,\infty}^{1/2-}([0,T],\mathbb{R}^{m})$ specified by requiring that
\[
\Psi(Y)_{t} = \null_{0}X_{t} + \int_{r}^{r\vee t}b(s,Y)\,ds + \int_{r}^{r\vee t}\sigma(s,Y)\,dW_{s}
\]
for all $t\in [0,T]$ a.s. We stress the fact that, due to Proposition~\ref{Kolmogorov-Chentsov Proposition}, for every $X\in\mathscr{H}$, $p > 4$ and $\alpha\in [0,1/2-2/p)$ it follows that
\[
E[\|\Psi(X)^{t}\|_{\alpha,r}^{p}]\leq c_{\alpha,p} E[\|\null_{0}X\|_{\infty}^{p}] + c_{\alpha,p}c_{p}\int_{r}^{t}1 + E[\|X^{s}\|_{\alpha,r}^{p}]\,ds
\]
for each $t\in [r,T]$. Thus, $\Psi(\mathscr{H})\subset\mathscr{H}$ follows from plugging~\eqref{General SDE Inequality 1} into the above inequality. Since $\null_{n}X = \Psi(\null_{n-1}X)$ a.s.~for all $n\in\mathbb{N}$, we have shown that $(\null_{n}X)_{n\in\mathbb{N}_{0}}$ is a sequence in $\mathscr{H}$.

Next, we choose $\alpha\in [0,1/2)$ and $p > 4$ such that $\alpha_{0}\leq \alpha < 1/2 - 2/p$, where $\alpha_{0}$ is the constant in~\eqref{C.4}. We set $\overline{c}_{p}:=(\int_{r}^{T}\lambda_{0}(s)^{2}\,ds)^{p/2} + w_{p}\tilde{\lambda}_{0}^{p}$, then it follows from Proposition~\ref{Kolmogorov-Chentsov Proposition} that any $X,\tilde{X}\in\mathscr{H}$ satisfy
\[
E[\|\Psi(X)^{t} - \Psi(\tilde{X})^{t}\|_{\alpha,r}^{p}] \leq c_{\alpha,p}\overline{c}_{p}\int_{r}^{t}E[\|X^{s} - \tilde{X}^{s}\|_{\alpha,r}^{p}]\,ds
\]
for all $t\in [r,T]$, since we can use that $\|x\|_{\alpha_{0},r}\leq \max\{1,T-r\}^{\alpha-\alpha_{0}}\|x\|_{\alpha,r}$ for every $x\in C_{r}^{\alpha}([0,T],\mathbb{R}^{m})$. Hence, Gronwall's inequality entails that there is at most a unique solution $X$ to~\eqref{SDE} such that $X^{r} = \hat{x}^{r}$ a.s. 

We infer from the above inequality that $\Psi$ is Lipschitz continuous with respect to the seminorm~\eqref{General SDE Seminorm}. In addition, 
\[
E[\|\null_{n+1}X^{t} - \null_{n}X^{t}\|_{\alpha,r}^{p}]\leq \frac{\delta^{p}}{n!}(c_{\alpha,p}\overline{c}_{p})^{n}(t-r)^{n}
\]
for each $n\in\mathbb{N}_{0}$ and $t\in [r,T]$, by induction with $\delta:=(E[\|\Psi(\null_{0}X) - \null_{0}X\|_{\alpha,r}^{p}])^{1/p}$. Hence, the triangle inequality gives
\[
\big(E\big[\|\null_{n}X - \null_{k}X\|_{\alpha,r}^{p}\big]\big)^{1/p} \leq \delta \sum_{i=k}^{n-1}\bigg(\frac{1}{i!}\bigg)^{1/p}(c_{\alpha,p}\overline{c}_{p})^{i/p}(T-r)^{i/p}
\]
for all $k,n\in\mathbb{N}_{0}$ with $k < n$. The ratio test shows that the series $\sum_{i=0}^{\infty}(1/i!)^{1/p}u^{i/p}$ converges absolutely for any $u\geq 0$. So, $\lim_{k\uparrow\infty}\sup_{n\in\mathbb{N}:\,n\geq k} E[\|\null_{n}X - \null_{k}X\|_{\alpha,r}^{p}] = 0$.

Due to Proposition~\ref{Hoelder Completeness Proposition 1}, there exists, up to indistinguishability, a unique process $X\in\mathscr{H}$ such that
\begin{equation}\label{General SDE Limit Equation}
\lim_{n\uparrow\infty} E[\|\null_{n}X - X\|_{\alpha,r}^{p}] = 0.
\end{equation}
Lipschitz continuity of $\Psi$ implies that $\lim_{n\uparrow\infty} E[\|\null_{n+1}X - \Psi(X)\|_{\alpha,r}^{p}] = 0$. For this reason, $X = \Psi(X)$ a.s. As $\alpha$ and $p$ have been arbitrarily chosen,~\eqref{General SDE Limit Equation} must hold for any $\alpha\in [0,1/2)$ and $p > 4$ such that $\alpha_{0}\leq \alpha < 1/2 - 2/p$.

Finally, if $\alpha\in [0,1/2)$ and $p\geq 1$ are such that $\alpha_{0}\leq \alpha < 1/2 - 2/p$ fails, then~\eqref{General SDE Limit Equation} is still true. Indeed, when proving this fact, we may, if necessary, replace $\alpha$ by $\alpha_{0}$ to ensure that $\alpha_{0}\leq \alpha$ is valid, since we have
\[
\|x\|_{\alpha,r}\leq \max\{1,T-r\}^{\alpha_{0}-\alpha}\|x\|_{\alpha_{0},r}\quad\text{for all $x\in C_{r}^{\alpha_{0}}([0,T],\mathbb{R}^{m})$}
\]
whenever $\alpha_{0} > \alpha$. Next, if $\alpha_{0}\leq \alpha$ but $\alpha$ $\geq 1/2 - 2/p$, we take $q > \max\{p,4\}$ with $\alpha < 1/2 - 2/q$ and use that $E[\|\null_{n}X-X\|_{\alpha,r}^{p}]$ $\leq (E[\|\null_{n}X-X\|_{\alpha,r}^{q}])^{p/q}$ for all $n\in\mathbb{N}$ to infer the desired result, which completes the proof.
\end{proof}

\section{Proof of the main result}\label{Proof of the main result}

\subsection{Decomposition into remainder terms}\label{Decomposition into remainder terms}

Let us deduce a moment estimate for any solution to~\eqref{General Sequence SDE} that is independent of $n\in\mathbb{N}$, which generalizes Proposition 3.1 in~\cite{SupportThm}.

\begin{proposition}\label{General Convergence Proposition 1}
Let~\eqref{C.6} be valid, $h\in H_{r}^{1}([0,T],\mathbb{R}^{d})$ and $\overline{B}$ be $d_{\infty}$-Lipschitz continuous. Then for each $p\geq 2$ there is $c_{p} > 0$ such that any $n\in\mathbb{N}$ and each solution $\null_{n}Y$ to~\eqref{General Sequence SDE} satisfy
\begin{equation}\label{General Convergence Inequality 1}
E[\|\null_{n}Y\|_{\infty}^{p}] + E[\|\null_{n}Y^{s} - \null_{n}Y^{t}\|_{\infty}^{p}]/|s-t|^{p/2}\leq c_{p}\big(1 + E[\|\null_{n}Y^{r}\|_{\infty}^{p}]\big)
\end{equation}
for all $s,t\in [r,T]$ with $s\neq t$.
\end{proposition}

\begin{proof}
We may assume that $E[\|\null_{n}Y^{r}\|_{\infty}^{p}]$ is finite and the constant $\kappa$ appearing in~\eqref{C.6} is positive. For $l\in\mathbb{N}$ the stopping time $\tau_{l,n}:=\inf\{t\in [0,T]\,|\, |\null_{n}Y_{t}|\geq l\}\vee r$ satisfies $\|\null_{n}Y^{\tau_{l,n}}\|_{\infty}\leq \|\null_{n}Y^{r}\|_{\infty}\vee l$ and for $s,t\in [r,T]$ with $s\leq t$ we have
\begin{equation}\label{Hoelder Main Estimation}
\begin{split}
\big(E[\|\null_{n}Y^{s\wedge\tau_{l,n}} - &\null_{n}Y^{t\wedge\tau_{l,n}}\|_{\infty}^{p}]\big)^{1/p}\\
&\leq \bigg(\overline{c}_{p}(t-s)^{p/2-1}\int_{s}^{t} 1 + E\big[\|\null_{n}Y^{u\wedge\tau_{l,n}}\|_{\infty}^{\kappa p}\big]\,du\bigg)^{1/p}\\
&\quad + \bigg(E\bigg[\sup_{v\in [s,t]}\bigg|\int_{s}^{v\wedge\tau_{l,n}}\overline{B}(u,\null_{n}Y)\,d\null_{n}W_{u}\bigg|^{p}\bigg]\bigg)^{1/p},
\end{split} 
\end{equation}
where $\overline{c}_{p}:=6^{p-1}((T-r)^{p/2} +\|h\|_{H,r}^{p} + w_{p})c^{p}$ and $w_{p}$ is the constant in~\eqref{Mao's Inequality}. Lemma~\ref{Brownian Interpolation Integral Moments Lemma} provides the constant $\hat{w}_{p/\kappa,1}$ such that \eqref{Brownian Interpolation Constant} holds when $p$ and $q$ are replaced by $p/\kappa$ and $1$, respectively. So,
\[
\bigg(E\bigg[\bigg(\int_{\underline{u}_{n}}^{u\wedge\tau_{l,n}}|\overline{B}(v,\null_{n}Y)\null_{n}\dot{W}_{v}|\,dv\bigg)^{p/\kappa}\bigg]\bigg)^{\kappa}\leq c_{p,1} (u-\underline{u}_{n})^{p/2}
\]
for any given $u\in [s,T]$ and $c_{p,1}:=2^{p/2}\hat{w}_{p/\kappa,1}^{\kappa}c^{p}$. By virtue of~\eqref{Hoelder Main Estimation}, we may define $\overline{c}_{p/\kappa}$ just as $\overline{c}_{p}$ above when $p$ is replaced by $p/\kappa$ and obtain that
\[
\big(E\big[\|\null_{n}Y^{u\wedge\tau_{l,n}} - \null_{n}Y^{\underline{u}_{n}\wedge\tau_{l,n}}\|_{\infty}^{p/\kappa}\big]\big)^{\kappa}\leq c_{p,2}(u-\underline{u}_{n})^{p/2}\big(1 + E\big[\|\null_{n}Y^{u\wedge\tau_{l,n}}\|_{\infty}^{p}\big]\big)^{\kappa}
\]
with $c_{p,2}:= 2^{p-1}(\overline{c}_{p/\kappa}^{\kappa} + c_{p,1})$. Consequently, by letting $\lambda\geq 0$ denote a Lipschitz constant for $\overline{B}$ and using \holders inequality, we can estimate that
\begin{align*}
&E\bigg[\bigg(\int_{s}^{t\wedge\tau_{l,n}}\big|\big(\overline{B}(u,\null_{n}Y) - \overline{B}(\underline{u}_{n},\null_{n}Y)\big)\null_{n}\dot{W}_{u}\big|\,du\bigg)^{p}\bigg]\\
&\leq (t-s)^{p/2-1}\int_{s}^{t}E\bigg[|\overline{B}(u,\null_{n}Y^{\tau_{l,n}}) - \overline{B}(\underline{u}_{n},\null_{n}Y^{\tau_{l,n}})|^{p}\bigg(\int_{s}^{t}|\null_{n}\dot{W}_{v}|^{2}\,dv\bigg)^{p/2}\bigg]\,du\\
&\leq c_{p,3}(t-s)^{p/2-1}\int_{s}^{t}\big(1 + E[\|\null_{n}Y^{u\wedge\tau_{l,n}}\|_{\infty}^{p}]\big)^{\kappa}\,du
\end{align*}
for $c_{p,3}:=2^{3p/2-1}\hat{w}_{(p/2)/(1-\kappa),2}^{1-\kappa}(T-r)^{p/2}\lambda^{p}(1 + c_{p,2})$. In addition, Proposition~\ref{Adapted Linear Interpolation Estimation} directly yields that
\[
E\bigg[\sup_{v\in [s,t]}\bigg|\int_{s}^{v\wedge\tau_{l,n}}\overline{B}(\underline{u}_{n},\null_{n}Y)\,d\null_{n}W_{u}\bigg|^{p}\bigg]\leq \hat{w}_{p} c^{p}(t-s)^{p/2}.
\]
Hence, we set $c_{p,4}:=3^{p-1}(2\overline{c}_{p} + c_{p,3} + \hat{w}_{p}c^{p})$, then from~\eqref{Hoelder Main Estimation} we in total obtain that
\begin{equation}\label{Hoelder Main Estimation 2}
E[\|\null_{n}Y^{s\wedge\tau_{l,n}} - \null_{n}Y^{t\wedge\tau_{l,n}}\|_{\infty}^{p}] \leq c_{p,4}(t-s)^{p/2-1}\int_{s}^{t} 1 + E[\|\null_{n}Y^{u\wedge\tau_{l,n}}\|_{\infty}^{p}]\,du.
\end{equation}
Now Gronwall's inequality and Fatou's lemma entail that
\[
E[\|\null_{n}Y^{t}\|_{\infty}^{p}]\leq\liminf_{l\uparrow\infty} E[\|\null_{n}Y^{t\wedge\tau_{l,n}}\|_{\infty}^{p}]\leq c_{p,5}\big(1 + E[\|\null_{n}Y^{r}\|_{\infty}^{p}]\big),
\]
where $c_{p,5}:=2^{p-1}\max\{1,T-r\}^{p/2}\max\{1,c_{p,4}\}e^{2^{p-1}(T-r)^{p/2}c_{p,4}}$. Thus, if we set $c_{p}:=(1 + c_{p,4})(1 + c_{p,5})$, then the claim follows from~\eqref{Hoelder Main Estimation 2} and another application of Fatou's lemma.
\end{proof}

\begin{corollary}\label{General Convergence Corollary 2}
Let~\eqref{C.6} hold and $h\in H_{r}^{1}([0,T],\mathbb{R}^{d})$. Then for each $p\geq 2$ there is $c_{p} > 0$ such that any solution $Y$ to~\eqref{General Limit SDE} satisfies
\begin{equation}\label{General Convergence Inequality 2}
E[\|Y\|_{\infty}^{p}] + E[\|Y^{s} - Y^{t}\|_{\infty}^{p}]/|s-t|^{p/2}\leq c_{p}\big(1 + E[\|Y^{r}\|_{\infty}^{p}]\big)
\end{equation}
for every $s,t\in [r,T]$ with $s\neq t$.
\end{corollary}

\begin{proof}
Since the map $R$ defined via~\eqref{General Remainder Term} is bounded, we may apply Proposition~\ref{General Convergence Proposition 1} in the case that $\underline{B}$ is replaced by $\underline{B} + R$, $\overline{B}$ is replaced by $0$ and $\Sigma$ is replaced by $\overline{B} + \Sigma$. From this the claim follows immediately.
\end{proof}

With the derived moment estimates we deduce a crucial decomposition estimate that involves the linear operator $L_{n}$ and the function $\gamma_{n}$ given at~\eqref{Delayed Linear Interpolation Operator} and~\eqref{Gamma Function}, respectively, where $n\in\mathbb{N}$.

\begin{proposition}\label{General Convergence Proposition 3}
Let~\eqref{C.6} and~\eqref{C.7} be satisfied and $h\in H_{r}^{1}([0,T],\mathbb{R}^{d})$. Then there is $c_{2} > 0$ such that for each $n\in\mathbb{N}$ and any solutions $\null_{n}Y$ and $Y$ to~\eqref{General Sequence SDE} and~\eqref{General Limit SDE}, respectively,
\begin{align*}
&E\big[\max_{j\in\{0,\dots,k_{n}\}}|\null_{n}Y_{t_{j,n}} - Y_{t_{j,n}}|^{2}\big]/c_{2}\leq |\mathbb{T}_{n}|\big(1 + E\big[\|\null_{n}Y^{r}\|_{\infty}^{2} + \|Y^{r}\|_{\infty}^{2}\big]\big)\\
&\quad + E\big[\|\null_{n}Y^{r} - Y^{r}\|_{\infty}^{2} + \|L_{n}(\null_{n}Y) - \null_{n}Y\|_{\infty}^{2} + \|L_{n}(Y) - Y\|_{\infty}^{2}\big]\\
&\quad + E\bigg[\max_{j\in\{0,\dots,k_{n}\}}\bigg|\int_{r}^{t_{j,n}}R(\underline{s}_{n},\null_{n}Y)(\gamma_{n}(s) - 1)\,ds\bigg|^{2}\bigg]\\
&\quad + E\bigg[\max_{j\in\{0,\dots,k_{n}\}}\bigg|\int_{r}^{t_{j,n}}\overline{B}(\underline{s}_{n},\null_{n}Y)\,d(\null_{n}W_{s} - W_{s})\bigg|^{2}\bigg]\\
&\quad + E\bigg[\max_{j\in\{0,\dots,k_{n}\}}\bigg|\int_{r}^{t_{j,n}}\big(\overline{B}(s,\null_{n}Y)-\overline{B}(\underline{s}_{n},\null_{n}Y)\big)\null_{n}\dot{W}_{s} - R(\underline{s}_{n},\null_{n}Y)\gamma_{n}(s)\,ds\bigg|^{2}\bigg].
\end{align*}
\end{proposition}

\begin{proof}
Let $E[\|\null_{n}Y^{r}\|_{\infty}^{2}]$ and $E[\|Y^{r}\|_{\infty}^{2}]$ be finite. We define an increasing function $\varphi_{n}:[r,T]\rightarrow\mathbb{R}_{+}$ by
\[
\varphi_{n}(t) := E\big[\max_{j\in\{0,\dots,k_{n}\}:\, t_{j,n}\leq t} |\null_{n}Y_{t_{j,n}} - Y_{t_{j,n}}|^{2}\big] 
\]
and seek to apply Gronwall's inequality. For this purpose, we write the difference of $\null_{n}Y$ and $Y$ in the form
\begin{align*}
\null_{n}Y_{t} - Y_{t} &= \null_{n}Y_{r} - Y_{r} +  \int_{r}^{t}\underline{B}(s,\null_{n}Y) - \underline{B}(s,Y)+ \big(B_{H}(s,\null_{n}Y) - B_{H}(s,Y)\big)\dot{h}(s)\,ds\\
&\quad + \null_{n}\Delta_{t} + \int_{r}^{t}\Sigma(s,\null_{n}Y) - \Sigma(s,Y)\,dW_{s}
\end{align*}
for all $t\in [r,T]$ a.s., where the process $\null_{n}\Delta\in\mathscr{C}([0,T],\mathbb{R}^{m})$ is chosen such that
\begin{equation*}
\null_{n}\Delta_{t}=\int_{r}^{t}\overline{B}(s,\null_{n}Y)\null_{n}\dot{W}_{s} - R(s,Y)\,ds - \int_{r}^{t}\overline{B}(s,Y)\,dW_{s}
\end{equation*}
for each $t\in [r,T]$ a.s. Hence, let $\lambda\geq 0$ denote a Lipschitz constant for $\underline{B}(s,\cdot)$, $B_{H}$, $\overline{B}$, $R$ and $\Sigma$ for any $s\in [r,T)$, then we obtain that
\begin{equation}\label{Gronwall Estimation I}
\varphi_{n}(t)^{1/2}\leq \delta_{n,1}^{1/2} + \delta_{n}(t)^{1/2} + \bigg(c_{2,1}\int_{r}^{t_{n}}\delta_{n,1} + \delta_{n,2}(s) +\varepsilon_{n}(s) + \varphi_{n}(s)\,ds\bigg)^{1/2}
\end{equation}
for every $t\in [r,T]$, where we have set $c_{2,1} := 15(T-r + \|h\|_{H,r}^{2} + w_{2})\lambda^{2}$ and $\delta_{n,1}:=E[\|\null_{n}Y^{r} - Y^{r}\|_{\infty}^{2}]$ and the functions $\delta_{n},\delta_{n,2},\varepsilon_{n}:[r,T]\rightarrow\mathbb{R}_{+}$, which are readily seen to be measurable, are defined via
\begin{align*}
\delta_{n}(t)&:=E\big[\max_{j\in\{0,\dots,k_{n}\}:\,t_{j,n}\leq t} |\null_{n}\Delta_{t_{j,n}}|^{2}\big],\\
\delta_{n,2}(s)&:=E\big[\|L_{n}(\null_{n}Y)^{\underline{s}_{n}}- \null_{n}Y^{\underline{s}_{n}}\|_{\infty}^{2} + \|L_{n}(Y)^{\underline{s}_{n}} - Y^{\underline{s}_{n}}\|_{\infty}^{2}\big]\quad\text{and}\\
\varepsilon_{n}(s) &:= E\big[\|\null_{n}Y^{s} - \null_{n}Y^{\underline{s}_{n}}\|_{\infty}^{2} + \|Y^{s} - Y^{\underline{s}_{n}}\|_{\infty}^{2}\big].
\end{align*}
In deriving~\eqref{Gronwall Estimation I} we used that $E[\|L_{n}(\null_{n}Y)^{\underline{s}_{n}} - L_{n}(Y)^{\underline{s}_{n}}\|_{\infty}^{2}]\leq \delta_{n,1} + \varphi_{n}(s)$ for all $s\in [r,T]$, which follows from Lemma~\ref{Delayed Linear Interpolation Lemma}, since $L_{n}$ is linear. 

To estimate $\delta_{n}$ we introduce three processes $\null_{n,3}\Delta,\null_{n,4}\Delta,\null_{n,5}\Delta\in\mathscr{C}([0,T],\mathbb{R}^{m})$ by setting $\null_{n,3}\Delta_{t}:=\int_{r}^{t}R(\underline{s}_{n},\null_{n}Y)(\gamma_{n}(s)-1)\,ds$ and
\begin{equation*}
\null_{n,5}\Delta_{t}:=\int_{r}^{t}\big(\overline{B}(s,\null_{n}Y) - \overline{B}(\underline{s}_{n},\null_{n}Y)\big)\null_{n}\dot{W}_{s} - R(\underline{s}_{n},\null_{n}Y)\gamma_{n}(s)\,ds
\end{equation*}
and requiring that $\null_{n,4}\Delta_{t} =\int_{r}^{t}\overline{B}(\underline{s}_{n},\null_{n}Y)\,d(\null_{n}W_{s} - W_{s})$ for each $t\in [r,T]$ a.s. Then $\null_{n}\Delta$ can be rewritten in the following way:
\begin{align*}
\null_{n}\Delta_{t}& = \null_{n,3}\Delta_{t} + \null_{n,4}\Delta_{t} + \null_{n,5}\Delta_{t} + \int_{r}^{t}R(\underline{s}_{n},\null_{n}Y) - R(s,Y)\,ds\\
&\quad + \int_{r}^{t}\overline{B}(\underline{s}_{n},\null_{n}Y) - \overline{B}(s,Y)\,dW_{s}
\end{align*}
for all $t\in [r,T]$ a.s. Thus, we set $c_{2,2}:= 10(T-r + w_{2})\lambda^{2}$, then it follows readily that
\begin{equation}\label{Gronwall Estimation II}
\begin{split}
\delta_{n}(t)^{1/2}&\leq \delta_{n,3}(t)^{1/2} + \delta_{n,4}(t)^{1/2} + \delta_{n,5}(t)^{1/2}\\
&\quad + \bigg(c_{2,2}\int_{r}^{t_{n}}\delta_{n,1} + (s-\underline{s}_{n}) + \delta_{n,2}(s) + \varepsilon_{n}(s)+ \varphi_{n}(s)\,ds\bigg)^{1/2}
\end{split}
\end{equation}
for each $t\in [r,T]$, where the increasing function $\delta_{n,i}:[r,T]\rightarrow\mathbb{R}_{+}$ is given by
\[
\delta_{n,i}(t):=E\big[\max_{j\in\{0,\dots,k_{n}\}:\,t_{j,n}\leq t}|\null_{n,i}\Delta_{t_{j,n}}|^{2}\big],\quad\text{for all $i\in\{3,4,5\}$.}
\]

Proposition~\ref{General Convergence Proposition 1} and Corollary~\ref{General Convergence Corollary 2} give two constants $\underline{c}_{2},\overline{c}_{2} > 0$ satisfying \eqref{General Convergence Inequality 1} and~\eqref{General Convergence Inequality 2} for $p=2$ when $c_{p}$ is replaced by $\underline{c}_{2}$ and $\overline{c}_{2}$, respectively. Thus, putting \eqref{Gronwall Estimation I} and~\eqref{Gronwall Estimation II} together, we find that
\begin{align*}
\varphi_{n}(t) &\leq c_{2,4}|\mathbb{T}_{n}|\big(1 + E\big[\|\null_{n}Y^{r}\|_{\infty}^{2} + \|Y^{r}\|_{\infty}^{2}\big]\big) + (5 + c_{2,3}(T-r))\delta_{n,1}\\
&\quad + 5\big(\delta_{n,3}(t) + \delta_{n,4}(t) + \delta_{n,5}(t)\big) + c_{2,3}\int_{r}^{t_{n}}\delta_{n,2}(s) + \varphi_{n}(s)\,ds
\end{align*}
for given $t\in [r,T]$, where we have at first set $c_{2,3}:=10(c_{2,1}+c_{2,2})$ and then $c_{2,4}:=2(T-r)(1 + \underline{c}_{2} + \overline{c}_{2})c_{2,3}$. Consequently,
\begin{align*}
\varphi_{n}(t)/c_{2}\leq |\mathbb{T}_{n}|\big(1 + E\big[\|\null_{n}Y^{r}\|_{\infty}^{2} + \|Y^{r}\|_{\infty}^{2}\big]\big) + \delta_{n,1} + \sum_{i=2}^{5}\delta_{n,i}(t)
\end{align*}
with $c_{2}:=e^{(T-r)c_{2,3}}(5 + c_{2,4})$, by Gronwall's inequality. This yields the claim.
\end{proof}

Thanks to Lemmas~\ref{Auxiliary Convergence Result 2} and~\ref{Auxiliary Convergence Result 3} and Proposition~\ref{Auxiliary Convergence Result 4}, only the last remainder in the estimation of Proposition~\ref{General Convergence Proposition 3} requires further analysis, before we can prove~\eqref{General SDE Partition Limit Equation}. For this reason, we define a map $\Phi_{h,n}:[r,T]\times C([0,T],\mathbb{R}^{m})\times C([0,T],\mathbb{R}^{d})\rightarrow\mathbb{R}^{m}$ by
{\begin{align*}
\Phi_{h,n}(s,y,w)&:= B_{H}(\underline{s}_{n},y)(h(s) - h(\underline{s}_{n})) + \overline{B}(\underline{s}_{n},y)\big(L_{n}(w)(s) - L_{n}(w)(\underline{s}_{n})\big)\\
&\quad + \Sigma(\underline{s}_{n},y)(w(s) - w(\underline{s}_{n}))
\end{align*}
for given $h\in H_{r}^{1}([0,T],\mathbb{R}^{d})$ and $n\in\mathbb{N}$. If now $\null_{n}Y$ is a solution to~\eqref{General Sequence SDE}, then the following decomposition can be used to deal with the remainder in question:
\begin{equation}\label{Main Remainder Decomposition}
\begin{split}
&\big(\overline{B}(s,\null_{n}Y) - \overline{B}(\underline{s}_{n},\null_{n}Y)\big)\null_{n}\dot{W}_{s} - R(\underline{s}_{n},\null_{n}Y)\gamma_{n}(s)\\
 &= \big(\overline{B}(s,\null_{n}Y) - \overline{B}(\underline{s}_{n},\null_{n}Y) - \partial_{x}\overline{B}(\underline{s}_{n},\null_{n}Y)(\null_{n}Y_{s} - \null_{n}Y_{\underline{s}_{n}})\big)\null_{n}\dot{W}_{s}\\
&\quad + \partial_{x}\overline{B}(\underline{s}_{n},\null_{n}Y)(\null_{n}Y_{s} - \null_{n}Y_{\underline{s}_{n}} -\Phi_{h,n}(s,\null_{n}Y,W))\null_{n}\dot{W}_{s}\\
&\quad + \partial_{x}\overline{B}(\underline{s}_{n},\null_{n}Y)\Phi_{h,n}(s,\null_{n}Y,W)\null_{n}\dot{W}_{s} - R(\underline{s}_{n},\null_{n}Y)\gamma_{n}(s)
\end{split}
\end{equation}
for any $s\in [r,T)$. In fact, in the next two sections we will consider each of these three terms to ensure that~\eqref{General SDE Partition Limit Equation} follows.

\subsection{Convergence of the first two remainders}\label{Convergence of the first two remainders}

To handle the first remainder term in~\eqref{Main Remainder Decomposition}, we will use the following estimation in combination with Lemma~\ref{Auxiliary Convergence Result 1}.

\begin{proposition}\label{Main Remainder Proposition 1}
Let~\eqref{C.6} be satisfied, $h\in H_{r}^{1}([0,T],\mathbb{R}^{d})$ and $F$ be a functional on $[r,T)\times C([0,T],\mathbb{R}^{m})$ of class $\mathbb{C}^{1,2}$. Assume that $\overline{B}$ and $\partial_{x}F$ are $d_{\infty}$-Lipschitz continuous and there are $c_{0},\eta\geq 0$ such that
\[
|\partial_{t}F(t,x)| + |\partial_{xx}F(t,x)|\leq c_{0}(1 + \|x\|_{\infty}^{\eta})
\]
for all $t\in [r,T)$ and $x\in C([0,T],\mathbb{R}^{m})$. Then for each $p\geq 2$ there exists $c_{p} > 0$ such that for each $n\in\mathbb{N}$ and any solution $\null_{n}Y$ to~\eqref{General Sequence SDE} it holds that
\begin{align*}
\sup_{s\in [r,T)}E\big[\big|F(s,\null_{n}Y) - F(\underline{s}_{n},\null_{n}Y) - &\partial_{x}F(\underline{s}_{n},\null_{n}Y)(\null_{n}Y_{s} - \null_{n}Y_{\underline{s}_{n}})\big|^{p}\big]\\
&\leq c_{p}|\mathbb{T}_{n}|^{p}\big(1 + E\big[\|\null_{n}Y^{r}\|_{\infty}^{(\eta\vee 2)p}\big]\big).
\end{align*}
\end{proposition}

\begin{proof}
Let $\null_{n}\Delta:[\underline{s}_{n},s]\times\Omega\rightarrow\mathbb{R}^{1\times m}$ be defined by $\null_{n}\Delta_{u} := \partial_{x}F(u,\null_{n}Y) - \partial_{x}F(\underline{s}_{n},\null_{n}Y)$ for fixed $s\in [r,T)$, then the functional It{\^o} formula in~\cite{ItoFormula} yields that
\begin{equation}\label{Functional Ito Formula Decomposition}
\begin{split}
&F(s,\null_{n}Y) - F(\underline{s}_{n},\null_{n}Y) - \partial_{x}F(\underline{s}_{n},\null_{n}Y)(\null_{n}Y_{s} - \null_{n}Y_{\underline{s}_{n}})\\
&= \int_{\underline{s}_{n}}^{s}\partial_{u}F(u,\null_{n}Y) + \null_{n}\Delta_{u}\big(\underline{B}(u,\null_{n}Y) + B_{H}(u,\null_{n}Y)\dot{h}(u) + \overline{B}(u,\null_{n}Y)\null_{n}\dot{W}_{u}\big)\,du\\
&\quad + \int_{\underline{s}_{n}}^{s}\null_{n}\Delta_{u}\Sigma(u,\null_{n}Y)\,dW_{u} + \frac{1}{2}\int_{\underline{s}_{n}}^{s}\mathrm{tr}(\partial_{xx}F(u,\null_{n}Y)(\Sigma\Sigma')(u,\null_{n}Y))\,du\quad\text{a.s.}
\end{split}
\end{equation}
By Proposition~\ref{General Convergence Proposition 1}, for $\overline{\eta}:=\eta\vee 2$ there is a constant $\underline{c}_{\overline{\eta}p} > 0$ such that~\eqref{General Convergence Inequality 1} holds when $c_{p}$ and $p$ are replaced by $\underline{c}_{\overline{\eta}p}$ and $\overline{\eta}p$, respectively. So,
\[
E\bigg[\bigg|\int_{\underline{s}_{n}}^{s}\partial_{u}F(u,\null_{n}Y)\,du\bigg|^{p}\bigg]\leq c_{p,1} |\mathbb{T}_{n}|^{p}\big(1 + E\big[\|\null_{n}Y^{r}\|_{\infty}^{\overline{\eta} p}\big]\big)^{\eta/\overline{\eta}}
\]
with $c_{p,1}:= 2^{2p}c_{0}^{p}(1 + \underline{c}_{\overline{\eta}p})^{\eta/\overline{\eta}}$. Next, let $\lambda_{0}\geq 0$ denote a Lipschitz constant for $\partial_{x}F$, then we obtain that
\begin{equation*}
\begin{split}
\big(E\big[|\null_{n}\Delta_{u}|^{2p}\big]\big)^{1/2}&\leq \overline{c}_{p}|\mathbb{T}_{n}|^{p/2}\big(1 + E\big[\|\null_{n}Y^{r}\|_{\infty}^{\overline{\eta}p}\big]\big)^{1/\overline{\eta}}
\end{split}
\end{equation*}
for all $u\in [\underline{s}_{n},s]$ with $\overline{c}_{p}:=2^{3p/2}\lambda_{0}^{p}(1 + \underline{c}_{\overline{\eta}p})^{1/\overline{\eta}}$. Hence, for the second term in~\eqref{Functional Ito Formula Decomposition} the Cauchy-Schwarz inequality gives
\[
E\bigg[\bigg|\int_{\underline{s}_{n}}^{s}\null_{n}\Delta_{u}\underline{B}(u,\null_{n}Y)\,du\bigg|^{p}\bigg]\leq c_{p,2}|\mathbb{T}_{n}|^{p}(s-\underline{s}_{n})^{p/2}\big(1 + E\big[\|\null_{n}Y^{r}\|_{\infty}^{\overline{\eta}p}\big]\big)^{2/\overline{\eta}}
\]
with $c_{p,2}:=2^{3p/2}c^{p}(1+\underline{c}_{\overline{\eta}p})^{1/\overline{\eta}}\overline{c}_{p}$. Similarly, it follows that
\[
E\bigg[\bigg|\int_{\underline{s}_{n}}^{s}\null_{n}\Delta_{u}B_{H}(u,\null_{n}Y)\,dh(u)\bigg|^{p}\bigg]\leq c_{p,3}|\mathbb{T}_{n}|^{p}\big(1 + E\big[\|\null_{n}Y^{r}\|_{\infty}^{\overline{\eta}p}\big]\big)^{2/\overline{\eta}}
\]
for $c_{p,3}:= 2^{3p/2}\|h\|_{H,r}^{p}c^{p}(1 + \underline{c}_{\overline{\eta}p})^{1/\overline{\eta}}\overline{c}_{p}$. Lemma~\ref{Brownian Interpolation Integral Moments Lemma} yields the constant $\hat{w}_{p,2}$ such that~\eqref{Brownian Interpolation Constant} is valid when $q$ is replaced by $2$. Then
\begin{align*}
&E\bigg[\bigg|\int_{\underline{s}_{n}}^{s}\null_{n}\Delta_{u}\overline{B}(u,\null_{n}Y)\,d\null_{n}W_{u}\bigg|^{p}\bigg]\\
&\leq c^{p}(s-\underline{s}_{n})^{p-1}\int_{\underline{s}_{n}}^{s}\big(E\big[|\null_{n}\Delta_{u}|^{2p}\big]\big)^{1/2}\bigg(E\bigg[\bigg(\int_{\underline{s}_{n}}^{s}|\null_{n}\dot{W}_{v}|^{2}\,dv\bigg)^{p}\bigg]\bigg)^{1/2}\,du\\
&\leq c_{p,4}|\mathbb{T}_{n}|^{p}\big(1 + E\big[\|\null_{n}Y^{r}\|_{\infty}^{\overline{\eta}p}\big]\big)^{1/\overline{\eta}},
\end{align*}
by the Cauchy-Schwarz inequality, where $c_{p,4}:=2^{p}\hat{w}_{p,2}^{1/2}c^{p}\overline{c}_{p}$. Next, as $\Sigma$ cannot exceed the constant $c$, we directly get that
\[
E\bigg[\bigg|\int_{\underline{s}_{n}}^{s}\null_{n}\Delta_{u}\Sigma(u,\null_{n}Y)\,dW_{u}\bigg|^{p}\bigg]\leq c_{p,5}|\mathbb{T}_{n}|^{p}\big(1 + E\big[\|\null_{n}Y^{r}\|_{\infty}^{\overline{\eta}p}\big]\big)^{1/\overline{\eta}}
\]
for $c_{p,5}:=2^{p/2}w_{p}c^{p}\overline{c}_{p}$, where $w_{p}$ satisfies~\eqref{Mao's Inequality}. For the last term in~\eqref{Functional Ito Formula Decomposition} we define $c_{p,6}:= 2^{2p}c^{2p}c_{0}^{p}(1 + \underline{c}_{\overline{\eta}p})^{\eta/\overline{\eta}}$ and readily compute that
\begin{equation*}
E\bigg[\bigg|\int_{\underline{s}_{n}}^{s}\mathrm{tr}(\partial_{xx}F(u,\null_{n}Y)(\Sigma\Sigma')(u,\null_{n}Y))\,du\bigg|^{p}\bigg]\leq c_{p,6}|\mathbb{T}_{n}|^{p}\big(1 + E\big[\|\null_{n}Y^{r}\|_{\infty}^{\overline{\eta}p}\big]\big)^{\eta/\overline{\eta}}.
\end{equation*}
Thus, by setting $c_{p}:=6^{p-1}(c_{p,1} + (T-r)^{p/2}c_{p,2} + c_{p,3} + c_{p,4} + c_{p,5} + 2^{-p}c_{p,6})$, we obtain the asserted estimate.
\end{proof}

We come to the second remainder term arising in~\eqref{Main Remainder Decomposition}. As before, we derive an estimation that is necessary to apply Lemma~\ref{Auxiliary Convergence Result 1}.

\begin{lemma}\label{Main Remainder Lemma 2}
Let~\eqref{C.6} and~\eqref{C.7} hold and $h\in H_{r}^{1}([0,T],\mathbb{R}^{d})$. Then for each $p \geq 2$ there is $c_{p} > 0$ such that for any $n\in\mathbb{N}$ and every solution $\null_{n}Y$ to~\eqref{General Sequence SDE} we have
\[
\sup_{s\in [r,T]}\big[|\null_{n}Y_{s} - \null_{n}Y_{\underline{s}_{n}} - \Phi_{h,n}(s,\null_{n}Y,W)|^{p}\big]\leq c_{p}|\mathbb{T}_{n}|^{p}\big(1 + E\big[\|\null_{n}Y^{r}\|_{\infty}^{2p}\big]\big)^{1/2}.
\]
\end{lemma}

\begin{proof}
We apply Proposition~\ref{General Convergence Proposition 1} to get a constant $\underline{c}_{2p} > 0$ such that~\eqref{General Convergence Inequality 1} is satisfied when $c_{p}$ and $p$ are replaced by $\underline{c}_{2p}$ and $2p$, respectively. Then
\[
E\bigg[\bigg|\int_{\underline{s}_{n}}^{s}\underline{B}(u,\null_{n}Y)\,du\bigg|^{p}\bigg]\leq c_{p,1}|\mathbb{T}_{n}|^{p}\big(1 + E\big[\|\null_{n}Y^{r}\|_{\infty}^{2p}\big]\big)^{1/2}
\]
for given $s\in [r,T]$ and $c_{p,1}:=2^{2p}c^{p}(1 + \underline{c}_{2p})^{1/2}$. Let $\lambda\geq 0$ denote a Lipschitz constant for $B_{H},\overline{B}$ and $\Sigma$, then the Cauchy-Schwarz inequality allows us to estimate that
\[
E\bigg[\bigg|\int_{\underline{s}_{n}}^{s}B_{H}(u,\null_{n}Y) - B_{H}(\underline{s}_{n},\null_{n}Y)\,dh(u)\bigg|^{p}\bigg]\leq c_{p,2}|\mathbb{T}_{n}|^{p}\big(1 + E\big[\|\null_{n}Y^{r}\|_{\infty}^{2p}\big]\big)^{1/2}
\]
with $c_{p,2}:= 2^{2p}\|h\|_{H,r}^{p}\lambda^{p}(1 + \underline{c}_{2p})^{1/2}$. We recall the constant $\hat{w}_{p,2}$ constructed in Lemma~\ref{Brownian Interpolation Integral Moments Lemma} such that~\eqref{Brownian Interpolation Constant} holds when $q$ is replaced by $2$ and compute that
\begin{align*}
&E\bigg[\bigg|\int_{\underline{s}_{n}}^{s}\overline{B}(u,\null_{n}Y) - \overline{B}(\underline{s}_{n},\null_{n}Y)\,d\null_{n}W_{u}\bigg|^{p}\bigg]\\
&\leq \lambda^{p}(s-\underline{s}_{n})^{p/2}E\bigg[\big((s-\underline{s}_{n})^{1/2} + \|\null_{n}Y^{s} - \null_{n}Y^{\underline{s}_{n}}\|_{\infty}\big)^{p}\bigg(\int_{\underline{s}_{n}}^{s}|\null_{n}\dot{W}_{v}|^{2}\,dv\bigg)^{p/2}\bigg]\\
& \leq c_{p,3}|\mathbb{T}_{n}|^{p}\big(1 + E\big[\|\null_{n}Y^{r}\|_{\infty}^{2p}\big]\big)^{1/2},
\end{align*}
due to the Cauchy-Schwarz inequality, where $c_{p,3}:=2^{5p/2}\hat{w}_{p,2}^{1/2}\lambda^{p}(1 + \underline{c}_{2p})^{1/2}$. Finally, let also recall the constant $w_{p}$ in~\eqref{Mao's Inequality}, then
\begin{align*}
E\bigg[\bigg|\int_{\underline{s}_{n}}^{s}\Sigma(u,\null_{n}Y) - \Sigma(\underline{s}_{n},\null_{n}Y)\,dW_{u}\bigg|^{p}\bigg]\leq c_{p,4}|\mathbb{T}_{n}|^{p}\big(1 + E\big[\|\null_{n}Y^{r}\|_{\infty}^{2p}\big]\big)^{1/2}
\end{align*}
for $c_{p,4}:= 2^{2p}w_{p}\lambda^{p}(1 + \underline{c}_{2p})^{1/2}$. So, the definition $c_{p}:=4^{p-1}(c_{p,1} + \cdots + c_{p,4})$ concludes the proof.
\end{proof}

\subsection{Convergence of the third remainder}\label{Convergence of the third remainder}

As preparation, we infer an estimate from Doob's ${L}^{2}$-maximal inequality. To this end, let for the moment $\tilde{d}\in\mathbb{N}$ and $\mathbb{T}$ be a partition of $[r,T]$ that is of the form $\mathbb{T}$ $=\{t_{0},\dots,t_{k}\}$ with $k\in\mathbb{N}$ and $t_{0},\dots,t_{k}\in [r,T]$ such that $r=t_{0} < \cdots < t_{k} = T$.

\begin{lemma}\label{Doobs Submartingale Lemma}
For every $l\in\{1,\dots,\tilde{d}\}$ let $(\null_{l}U_{i})_{i\in\{1,\dots,k\}}$ and $(\null_{l}V_{i})_{i\in\{1,\dots,k\}}$ be two sequences of $\mathbb{R}^{1\times d}$-valued and $\mathbb{R}^{d}$-valued random vectors, respectively, such that $\null_{l}U_{i}$ is $\mathscr{F}_{t_{i-1}}$-measurable, $\null_{l}V_{i}$ is $\mathscr{F}_{t_{i}}$-measurable,
\[
E\big[|\null_{l}U_{i}|^{4} + |\null_{l}V_{i}|^{4}\big] < \infty\quad\text{and}\quad E[\null_{l}V_{i}|\mathscr{F}_{t_{i-1}}] = 0\quad\text{a.s.}
\]
for all $i\in\{1,\dots,k\}$. Then
\begin{align*}
E\bigg[\max_{j\in\{i_{0},\dots,k\}}\bigg|\sum_{i=1}^{j-i_{0}}\sum_{l=1}^{\tilde{d}}\null_{l}U_{i}\,\null_{l}V_{i}\bigg|^{2}\bigg]\leq 4\sum_{i=1}^{k-i_{0}}\sum_{l_{1},l_{2}=1}^{\tilde{d}} E\big[\null_{l_{1}}U_{i}\,\null_{l_{1}}V_{i}\,\null_{l_{2}}V_{i}'\null_{l_{2}}U_{i}'\big]
\end{align*}
for every $i_{0}\in\{0,\dots,k-1\}$.
\end{lemma}

\begin{proof}
We set $Y_{i}:= \sum_{l=1}^{\tilde{d}}\null_{l}U_{i}\,\null_{l}V_{i}$ for each $i\in\{1,\dots,k-i_{0}\}$, then the sequence $(S_{j})_{j\in\{i_{0},\dots,k\}}$ of random variables given by $S_{j}:=\sum_{i=1}^{j-i_{0}}Y_{i}$ is a square-integrable martingale with respect to $(\mathscr{F}_{t_{j-i_{0}}})_{j\in\{i_{0},\dots,k\}}$. So,
\begin{equation*}
E\bigg[\max_{j\in\{i_{0},\dots,k\}}\bigg|\sum_{i=1}^{j-i_{0}}\sum_{l=1}^{\tilde{d}}\null_{l}U_{i}\,\null_{l}V_{i}\bigg|^{2}\bigg] = E\big[\max_{j\in\{i_{0},\dots,k\}}S_{j}^{2}\big]\leq 4 E\big[S_{k}^{2}\big],
\end{equation*}
by Doob's ${L}^{2}$-maximal inequality. Moreover, let $i,j\in\{1,\dots,k-i_{0}\}$ be such that $i\leq j$, then
\begin{equation*}
E[Y_{i}Y_{j}] = \mathbbm{1}_{\{i\}}(j) \sum_{l_{1},l_{2}=1}^{\tilde{d}} E[\null_{l_{1}}U_{i}\,E[\null_{l_{1}}V_{i}\,\null_{l_{2}}V_{j}'|\mathscr{F}_{t_{j-1}}]\,\null_{l_{2}}U_{j}'].
\end{equation*}
In particular, $Y_{i}$ and $Y_{j}$ are uncorrelated for $i < j$. By Bienaym{\'e}'s identity, $E[S_{k}^{2}]$ $= \sum_{i=1}^{k-i_{0}}E[Y_{i}^{2}]$, which yields the claim.
\end{proof}

\begin{proposition}\label{Main Remainder Proposition 5}
Let~\eqref{C.6} and~\eqref{C.7} hold and $h\in H_{r}^{1}([0,T],\mathbb{R}^{d})$. Then there is $c_{2} > 0$ such that for each $n\in\mathbb{N}$ and any solution $\null_{n}Y$ to~\eqref{General Sequence SDE},
\begin{align*}
E\bigg[\max_{j\in\{0,\dots,k_{n}\}}\bigg|\int_{r}^{t_{j,n}}\partial_{x}\overline{B}(\underline{s}_{n},\null_{n}Y)&\Phi_{h,n}(s,\null_{n}Y,W)\null_{n}\dot{W}_{s} - R(\underline{s}_{n},\null_{n}Y)\gamma_{n}(s)\,ds\bigg|^{2}\bigg]\\
& \leq c_{2} |\mathbb{T}_{n}|\big(1 + E\big[\|\null_{n}Y^{r}\|_{\infty}^{2}\big]\big).
\end{align*}
\end{proposition}

\begin{proof}
First, let us recall the definition of $R$ in~\eqref{General Remainder Term} to write the $k$-th coordinate of $\partial_{x}\overline{B}(\underline{s}_{n},\null_{n}Y)\Phi_{h,n}(s,\null_{n}Y,W)\,\null_{n}\dot{W}_{s}$ $-\, R(\underline{s}_{n},\null_{n}Y)\gamma_{n}(s)$ in the form
\[
\sum_{l=1}^{d}\partial_{x}\overline{B}_{k,l}(\underline{s}_{n},\null_{n}Y)\big(\Phi_{h,n}(s,\null_{n}Y,W)\null_{n}\dot{W}_{s}^{(l)} - ((1/2)\overline{B} + \Sigma)(\underline{s}_{n},\null_{n}Y)\gamma_{n}(s)e_{l}\big)
\]
for all $k\in\{1,\dots,m\}$ and $s\in [r,T)$, where the $l$-th coordinate of any $\mathbb{R}^{d}$-valued stochastic process $X$ is denoted by $X^{(l)}$ for every $l\in\{1,\dots,d\}$. Moreover, we decompose that\begin{equation}\label{Integrand Decomposition}
\begin{split}
&\Phi_{h,n}(s,\null_{n}Y,W)\null_{n}\dot{W}_{s}^{(l)} - ((1/2)\overline{B} + \Sigma)(\underline{s}_{n},\null_{n}Y)\gamma_{n}(s)e_{l}\\
&= B_{H}(\underline{s}_{n},\null_{n}Y)(h(s_{n})-h(\underline{s}_{n}))\null_{n}\dot{W}_{s}^{(l)} + \overline{B}(\underline{s}_{n},\null_{n}Y)(\null_{n}W_{s_{n}}-\null_{n}W_{\underline{s}_{n}})\null_{n}\dot{W}_{s}^{(l)}\\
&\quad + \Sigma(\underline{s}_{n},\null_{n}Y)\big(\Delta W_{s_{n}}\null_{n}\dot{W}_{s}^{(l)} - \gamma_{n}(s)e_{l}\big) + B_{H}(\underline{s}_{n},\null_{n}Y)(h(s) - h(s_{n}))\null_{n}\dot{W}_{s}^{(l)}\\
&\quad + \overline{B}(\underline{s}_{n},\null_{n}Y)\big((\null_{n}W_{s} - \null_{n}W_{s_{n}})\null_{n}\dot{W}_{s}^{(l)} - (1/2)\gamma_{n}(s)e_{l}\big)\\
&\quad + \Sigma(\underline{s}_{n},\null_{n}Y)(W_{s} - W_{s_{n}})\null_{n}\dot{W}_{s}^{(l)}
\end{split}
\end{equation}
for any $l\in\{1,\dots,d\}$. We begin with the first term in this decomposition and use Lemma~\ref{Brownian Linear Interpolation Property Lemma} to obtain that
\begin{align*}
&\int_{r}^{t_{j,n}}(\partial_{x}\overline{B}_{k,l} B_{H})(\underline{s}_{n},\null_{n}Y)(h(s_{n})-h(\underline{s}_{n}))\null_{n}\dot{W}_{s}^{(l)}\,ds\\
&= \int_{r}^{t_{j-1,n}}(\partial_{x}\overline{B}_{k,l} B_{H})(s_{n},\null_{n}Y)(h(\overline{s}_{n})-h(s_{n}))\,dW_{s}^{(l)}\quad\text{a.s.}
\end{align*}
for each $j\in\{1,\dots,k_{n}\}$, $k\in\{1,\dots,m\}$ and $l\in\{1,\dots,d\}$. Proposition~\ref{General Convergence Proposition 1} provides $\underline{c}_{2} > 0$ such that~\eqref{General Convergence Inequality 1} is satisfied for $p=2$ when the appearing constant $c_{p}$ is replaced by $\underline{c}_{2}$. Hence, condition~\eqref{C.6} gives
\begin{align*}
E\bigg[\max_{j\in\{0,\dots,k_{n}\}}&\sum_{k=1}^{m}\bigg|\int_{r}^{t_{j,n}}\sum_{l=1}^{d}(\partial_{x}\overline{B}_{k,l} B_{H})(\underline{s}_{n},\null_{n}Y)(h(s_{n})-h(\underline{s}_{n}))\null_{n}\dot{W}_{s}^{(l)}\,ds\bigg|^{2}\bigg]\\
&\leq 2w_{2}c^{4}\int_{r}^{T}\big(1 + E\big[\|\null_{n}Y^{s_{n}}\|_{\infty}^{2\kappa}\big]\big)|h(\overline{s}_{n})-h(s_{n})|^{2}\,ds\\
&\leq c_{2,1}|\mathbb{T}_{n}|\big(1 + E\big[\|\null_{n}Y^{r}\|_{\infty}^{2}\big]\big),
\end{align*}
where $c_{2,1}:= 2^{2}w_{2}(T-r)\|h\|_{H,r}^{2}c^{4}(1+\underline{c}_{2})$ and $w_{2}$ satisfies~\eqref{Mao's Inequality} for $p=2$. Similarly, another application of Lemma~\ref{Brownian Linear Interpolation Property Lemma} gives us that
\begin{align*}
&\int_{r}^{t_{j,n}}(\partial_{x}\overline{B}_{k,l} \overline{B})(\underline{s}_{n},\null_{n}Y)(\null_{n}W_{s_{n}}-\null_{n}W_{\underline{s}_{n}})\null_{n}\dot{W}_{s}^{(l)}\,ds\\
&= \int_{r}^{t_{j-1,n}} (\partial_{x}\overline{B}_{k,l} \overline{B})(s_{n},\null_{n}Y)\Delta W_{s_{n}}\,dW_{s}^{(l)}\quad\text{a.s.}
\end{align*}
for all $j\in\{1,\dots,k_{n}\}$, $k\in\{1,\dots,m\}$ and $l\in\{1,\dots,d\}$. Thus, with the constant $c_{2,2}:= w_{2}(T-r)dc^{4}$ we can estimate that
\begin{align*}
E\bigg[\max_{j\in\{0,\dots,k_{n}\}}&\sum_{k=1}^{m}\bigg|\int_{r}^{t_{j,n}}\sum_{l=1}^{d}(\partial_{x}\overline{B}_{k,l} \overline{B})(\underline{s}_{n},\null_{n}Y)(\null_{n} W_{s_{n}}-\null_{n}W_{\underline{s}_{n}})\null_{n}\dot{W}_{s}^{(l)}\,ds\bigg|^{2}\bigg]\\
&\leq w_{2}c^{4}\int_{r}^{T}E\big[|\Delta W_{s_{n}}|^{2}\big]\,ds \leq c_{2,2}|\mathbb{T}_{n}|.
\end{align*}

Let us move on to the third expression in~\eqref{Integrand Decomposition}. First, we define an $\mathbb{R}^{d}$-valued $\mathscr{F}_{t_{i,n}}$-measurable random vector by
\begin{equation}\label{Auxiliary random vector}
\null_{l,n}V_{i}:= \Delta W_{t_{i,n}}\Delta W_{t_{i,n}}^{(l)} - \Delta t_{i,n}e_{l}
\end{equation}
for every $i\in\{1,\dots,k_{n}\}$ and $l\in\{1,\dots,d\}$, then $\null_{l,n}V_{i}$ is independent of $\mathscr{F}_{t_{i-1,n}}$ and satisfies $E[|\null_{l,n}V_{i}|^{4}] < \infty$ and $E[\null_{l,n}V_{i}] = 0$. Moreover, a case distinction shows that
\begin{equation}\label{Auxiliary random vector 2}
E[\null_{l_{1},n}V_{i}\,\null_{l_{2},n}V_{i}']= \mathbbm{1}_{\{l_{2}\}}(l_{1})(\Delta t_{i,n})^{2}\big(\mathbb{I}_{d} + \mathbb{I}_{l_{2},l_{1}}\big)
\end{equation}
for each $i\in\{1,\dots,k_{n}\}$ and $l_{1},l_{2}\in\{1,\dots,d\}$, where $\mathbb{I}_{l_{2},l_{1}}\in\mathbb{R}^{d\times d}$ denotes the matrix whose $(l_{2},l_{1})$-entry is $1$ and whose all other entries are zero. We compute that
\begin{align*}
&\int_{r}^{t_{j,n}}(\partial_{x}\overline{B}_{k,l}\Sigma)(\underline{s}_{n},\null_{n}Y)\big(\Delta W_{s_{n}}\null_{n}\dot{W}_{s}^{(l)} - \gamma_{n}(s) e_{l}\big)\,ds\\
&= \sum_{i=1}^{j-1}(\partial_{x}\overline{B}_{k,l}\Sigma)(t_{i-1,n},\null_{n}Y)\null_{l,n}V_{i}
\end{align*}
for all $j\in\{1,\dots,k_{n}\}$, $k\in\{1,\dots,m\}$ and $l\in\{1,\dots,d\}$, since $\gamma_{n}(s) = 0$ for each $s\in [r,t_{1,n})$. Consequently, Lemma~\ref{Doobs Submartingale Lemma} and the representation~\eqref{Auxiliary random vector 2} imply that
\begin{align*}
E\bigg[\max_{j\in\{0,\dots,k_{n}\}}&\sum_{k=1}^{m}\bigg|\int_{r}^{t_{j,n}}\sum_{l=1}^{d}(\partial_{x}\overline{B}_{k,l} \Sigma)(\underline{s}_{n},\null_{n}Y)\big(\Delta W_{s_{n}}\null_{n}\dot{W}_{s}^{(l)} - \gamma_{n}(s)e_{l}\big)\,ds\bigg|^{2}\bigg]\\
&\leq 2^{3}\sum_{i=1}^{k_{n}-1}(\Delta t_{i,n})^{2}\sum_{k=1}^{m}\sum_{l=1}^{d} E[|(\partial_{x}\overline{B}_{k,l}\Sigma)(t_{i-1,n},\null_{n}Y)|^{2}] \leq c_{2,3}|\mathbb{T}_{n}|
\end{align*}
for $c_{2,3}:= 2^{3}(T-r)c^{4}$, since we can use that $\overline{x}^{t}\mathbbm{I}_{l_{2},l_{1}}\overline{y} \leq (1/2)(\overline{x}_{l_{2}}^{2} + \overline{y}_{l_{1}}^{2})$ for all $l_{1},l_{2}\in\{1,\dots,d\}$ and $\overline{x},\overline{y}\in\mathbb{R}^{d}$, by Young's inequality. To deal with the fourth term in \eqref{Integrand Decomposition}, let us note that
\begin{equation}\label{First Expression Computation}
\begin{split}
&\int_{r}^{t_{j,n}}(\partial_{x}\overline{B}_{k,l} B_{H})(\underline{s}_{n},\null_{n}Y)(h(s) - h(s_{n}))\null_{n}\dot{W}_{s}^{(l)}\,ds\\
&=\sum_{i=1}^{j-1}(\partial_{x}\overline{B}_{k,l} B_{H})(t_{i-1,n},\null_{n}Y)\frac{\Delta W_{t_{i,n}}^{(l)}}{\Delta t_{i+1,n}}\int_{t_{i,n}}^{t_{i+1,n}}h(s) - h(t_{i,n})\,ds\\
&= \int_{r}^{t_{j,n}}(\partial_{x}\overline{B}_{k,l} B_{H})(\underline{s}_{n},\null_{n}Y)\Delta W_{s_{n}}^{(l)}\frac{(\overline{s}_{n} - s)}{\Delta \overline{s}_{n}}\,dh(s)
\end{split}
\end{equation}
for each $j\in\{1,\dots,k_{n}\}$, $k\in\{1,\dots,m\}$ and $l\in\{1,\dots,d\}$, as integration by parts yields that $\int_{t_{i,n}}^{t_{i+1,n}}h(s) - h(t_{i,n})\,ds = \int_{t_{i,n}}^{t_{i+1,n}}t_{i+1,n} - s\,dh(s)$ for all $i\in\{0,\dots,k_{n}-1\}$. So,
\begin{align*}
E\bigg[\max_{j\in\{0,\dots,k_{n}\}}&\sum_{k=1}^{m}\bigg|\int_{r}^{t_{j,n}}\sum_{l=1}^{d}(\partial_{x}\overline{B}_{k,l} B_{H})(\underline{s}_{n},\null_{n}Y)(h(s) - h(s_{n}))\null_{n}\dot{W}_{s}^{(l)}\,ds\bigg|^{2}\bigg]\\
&\leq \|h\|_{H,r}^{2}\sum_{k=1}^{m}\int_{r}^{T}E\bigg[\bigg|\sum_{l=1}^{d}(\partial_{x}\overline{B}_{k,l}B_{H})(\underline{s}_{n},\null_{n}Y)\Delta W_{s_{n}}^{(l)}\bigg|^{2}\bigg]\,ds\\
&\leq c_{2,4}|\mathbb{T}_{n}|\big(1 + E\big[\|\null_{n}Y^{r}\|_{\infty}^{2}\big]\big)
\end{align*}
with $c_{2,4}:=2^{2}(T-r)\|h\|_{H,r}^{2}c^{4}(1+\underline{c}_{2})$, by the Cauchy-Schwarz inequality and the facts that $\Delta W_{s_{n}}^{(1)},\dots,\Delta W_{s_{n}}^{(d)}$ are pairwise independent and independent of $\mathscr{F}_{\underline{s}_{n}}$ for all $s\in [r,T]$.

To handle the fifth expression in~\eqref{Integrand Decomposition}, we proceed similarly as with the third expression. We define $\null_{l,n}U_{s}:=(\null_{n}W_{s}-\null_{n}W_{s_{n}})\null_{n}\dot{W}_{s}^{(l)} - (1/2)\gamma_{n}(s)e_{l}$ for all $s\in [r,T]$ and note that
\[
\int_{r}^{t_{j,n}}(\partial_{x}\overline{B}_{k,l}\overline{B})(\underline{s}_{n},\null_{n}Y)\null_{l,n}U_{s}\,ds = \frac{1}{2}\sum_{i=1}^{j-1}(\partial_{x}\overline{B}_{k,l} \overline{B})(t_{i-1,n},\null_{n}Y)\null_{l,n}V_{i}
\]
for every $j\in\{1,\dots,k_{n}\}$, $k\in\{1,\dots,m\}$ and $l\in\{1,\dots,d\}$, where we utilized that $\int_{t_{i,n}}^{t_{i+1,n}} s - t_{i,n}\,ds$ $= (1/2)(\Delta t_{i+1,n})^{2}$ for each $i\in\{0,\dots,k_{n}-1\}$. Consequently,
\begin{align*}
E\bigg[\max_{j\in\{0,\dots,k_{n}\}}&\sum_{k=1}^{m}\bigg|\int_{r}^{t_{j,n}}\sum_{l=1}^{d}(\partial_{x}\overline{B}_{k,l} \overline{B})(\underline{s}_{n},\null_{n}Y)\null_{l,n}U_{s}\,ds\bigg|^{2}\bigg]\\
&\leq 2\sum_{i=1}^{k_{n}-1}(\Delta t_{i,n})^{2}\sum_{k=1}^{m}\sum_{l=1}^{d}E[|(\partial_{x}\overline{B}_{k,l}\overline{B})(t_{i-1,n},\null_{n}Y)|^{2}]\leq c_{2,5}|\mathbb{T}_{n}|
\end{align*}
for $c_{2,5}:= 2(T-r)c^{4}$. We turn to the last term in~\eqref{Integrand Decomposition} and proceed just as in~\eqref{First Expression Computation} to get that
\begin{align*}
&\int_{r}^{t_{j,n}}(\partial_{x} \overline{B}_{k,l}\Sigma)(\underline{s}_{n},\null_{n}Y)(W_{s} - W_{s_{n}})\null_{n}\dot{W}_{s}^{(l)}\,ds\\
&= \int_{r}^{t_{j,n}}(\partial_{x}\overline{B}_{k,l} \Sigma)(\underline{s}_{n},\null_{n}Y)\Delta W_{s_{n}}^{(l)}\frac{(\overline{s}_{n} - s)}{\Delta \overline{s}_{n}}\,dW_{s}\quad\text{a.s.}
\end{align*}
for each $j\in\{1,\dots,k_{n}\}$, $k\in\{1,\dots,m\}$ and $l\in\{1,\dots,d\}$, as It{\^o}'s formula gives $\int_{t_{i,n}}^{t_{i+1,n}}W_{s} - W_{t_{i,n}}\,ds = \int_{t_{i,n}}^{t_{i+1,n}}t_{i+1,n} - s\,dW_{s}$ a.s.~for all $i\in\{0,\dots,k_{n}-1\}$. Therefore,
\begin{align*}
E\bigg[\max_{j\in\{0,\dots,k_{n}\}}&\sum_{k=1}^{m}\bigg|\int_{r}^{t_{j,n}}\sum_{l=1}^{d}(\partial_{x}\overline{B}_{k,l}\Sigma)(\underline{s}_{n},\null_{n}Y)(W_{s} - W_{s_{n}})\null_{n}\dot{W}_{s}^{(l)}\,ds\bigg|^{2}\bigg]\\
&\leq w_{2}\sum_{k=1}^{m}\int_{r}^{T}E\bigg[\bigg|\sum_{l=1}^{d}(\partial_{x}B_{k,l}\Sigma)(\underline{s}_{n},\null_{n}Y)\Delta W_{s_{n}}^{(l)}\bigg|^{2}\bigg]\,ds\leq c_{2,6}|\mathbb{T}_{n}|
\end{align*}
with $c_{2,6}:=w_{2}(T-r)c^{4}$. As before, we used that $\Delta W_{s_{n}}^{(1)},\dots,\Delta W_{s_{n}}^{(d)}$ are pairwise independent and independent of $\mathscr{F}_{\underline{s}_{n}}$ for every $s\in [r,T]$. Therefore, by setting $c_{2}:= 6(c_{2,1} + \cdots + c_{2,6})$, the assertion follows.
\end{proof}

\subsection{Proofs of Theorems~\ref{General SDE Convergence Theorem} and~\ref{Support Theorem}}\label{Proofs of Theorem 7 and 1}

At first, we consider a sufficient condition for a Dol{\'e}ans-Dade exponential to be a true martingale and let temporarily $\mathbb{T}$ be a partition of $[r,T]$ of the form $\mathbb{T}=\{t_{0},\dots,t_{k}\}$ with $k\in\mathbb{N}$ and $t_{0},\dots,t_{k}\in [r,T]$ satisfying $r=t_{0} < \dots < t_{k} = T$.

\begin{lemma}\label{Doleans-Dade Exponential Lemma}
Let $f:[r,T]\rightarrow\mathbb{R}^{m\times d}$ be measurable such that $\int_{r}^{T}|f(s)|^{2}\,ds < \infty$ and $(Y_{i})_{i\in\{0,\dots,k-1\}}$ be an $(\mathscr{F}_{t_{i}})_{i\in\{0,\dots,k-1\}}$-adapted sequence of $\mathbb{R}^{m\times d}$-valued random matrices. Define the process $X:[r,T]\times\Omega\rightarrow\mathbb{R}^{1\times d}$ coordinatewise via
\[
X_{t}^{(l)}:=\sum_{i=0}^{k-1}\sum_{j=1}^{m}f_{j,l}(t)Y_{i}^{(j,l)}\mathbbm{1}_{[t_{i},t_{i+1})}(t),
\]
then the continuous local martingale $Z\in\mathscr{C}([0,T],\mathbb{R})$ given by $Z^{r} = 1$ and
\[
Z_{t} = \exp\bigg(\int_{r}^{t}X_{s}\,dW_{s} - \frac{1}{2}\int_{r}^{t}|X_{s}|^{2}\,ds\bigg)
\]
for all $t\in [r,T]$ a.s.~is a martingale.
\end{lemma}

\begin{proof}
Since $Z$ is a positive supermartingale, it suffices to show that $E[Z_{T}]=1$. This in turn follows inductively if we can verify that $E[Z_{t_{i+1}}|\mathscr{F}_{t_{i}}] = Z_{t_{i}}$ a.s.~for each $i\in\{0,\dots,k-1\}$.

In this regard, note that $\int_{t_{i}}^{t_{i+1}}X_{s}\,dW_{s} = \sum_{l=1}^{d}\sum_{j=1}^{m} Y_{i}^{(j,l)}\int_{t_{i}}^{t_{i+1}}f_{j,l}(s)\,dW_{s}^{(l)}$ a.s. Because $\int_{t_{i}}^{t_{i+1}}f_{j,1}(s)\,dW_{s}^{(1)}$, $\dots$, $\int_{t_{i}}^{t_{i+1}}f_{j,d}(s)\,dW_{s}^{(d)}$ are independent of $\mathscr{F}_{t_{i}}$ for any $j\in\{1,\dots,m\}$, we have
\[
E\bigg[\exp\bigg(\int_{t_{i}}^{t_{i+1}}X_{s}\,dW_{s} - \frac{1}{2}\int_{t_{i}}^{t_{i+1}}|X_{s}|^{2}\,ds\bigg)\bigg|\mathscr{F}_{t_{i}}\bigg] = \psi(Y_{i})\quad\text{a.s.},
\]
where the Borel measurable function $\psi:\mathbb{R}^{m\times d}\rightarrow\mathbb{R}$ is given by
\[
\psi(A):= E\bigg[\prod_{l=1}^{d}\exp\bigg(\int_{t_{i}}^{t_{i+1}}\sum_{j=1}^{m}A_{j,l}f_{j,l}(s)\,dW_{s}^{(l)} - \frac{1}{2}\int_{t_{i}}^{t_{i+1}}\bigg|\sum_{j=1}^{m}A_{j,l}f_{j,l}(s)\bigg|^{2}\,ds\bigg)\bigg].
\]
Moreover, as $\int_{t_{i}}^{t_{i+1}}\sum_{j=1}^{m}A_{j,l}f_{j,l}(s)\,dW_{s}^{(l)}$ is normally distributed with zero mean and variance given by $\int_{t_{i}}^{t_{i+1}}|\sum_{j=1}^{m}A_{j,l}f_{j,l}(s)|^{2}\,ds$ for any $A\in\mathbb{R}^{m\times d}$ and $l\in\{1,\dots,d\}$, independence of the coordinates of $W$ entails that $\psi=1$, as desired.
\end{proof}

\begin{proof}[Proof of Theorem~\ref{General SDE Convergence Theorem}]
(i) If in condition~\eqref{C.8} we have $\overline{b}_{0} = 0$, then existence and uniqueness can be inferred from Proposition~\ref{General ODE Proposition}. Otherwise, we may let $\overline{b}_{0} = 1$ and, by using Lemma~\ref{Doleans-Dade Exponential Lemma}, define a martingale $\null_{n}\overline{Z}\in\mathscr{C}([0,T],\mathbb{R})$ via $\null_{n}\overline{Z}^{r} =1$ and
\[
\null_{n}\overline{Z}_{t}=\exp\bigg(-\int_{r}^{t}\overline{b}(s)\null_{n}\dot{W}_{s}'\,dW_{s} - \frac{1}{2}\int_{r}^{t}|\overline{b}(s)\null_{n}\dot{W}_{s}|^{2}\,ds\bigg)
\]
for all $t\in [r,T]$ a.s. Due to Girsanov's theorem, the process $\null_{n}\overline{W}\in\mathscr{C}([0,T],\mathbb{R}^{d})$ given by $
\null_{n}\overline{W}_{t}:= W_{t} + \int_{r}^{r\vee t}\overline{b}(s)\,d\null_{n}W_{s}$ is a $d$-dimensional $(\mathscr{F}_{t})_{t\in [0,T]}$-Brownian motion under the probability measure $\overline{P}_{n}$ on $(\Omega,\mathscr{F})$ defined by $\overline{P}_{n}(A):=E[\mathbbm{1}_{A}\null_{\,n}\overline{Z}_{T}]$. 

As this yields an equivalent probability measure, a process $Y\in\mathscr{C}([0,T],\mathbb{R}^{m})$ solves~\eqref{General Sequence SDE} under $P$ if and only if it is a solution to the SDE
\[
dY_{t} = \big(\underline{B}(t,Y) + B_{H}(t,Y)\dot{h}(t)\big)\,dt + \Sigma(t,Y)\,d\null_{n}\overline{W}_{t}\quad\text{for $t\in [r,T]$}
\]
under $\overline{P}_{n}$. For this reason, existence and uniqueness follow from Proposition~\ref{General SDE Proposition} when $b=\underline{B} + B_{H}\dot{h}$ and $\sigma=\Sigma$. Further, independently of this case distinction, Propositions~\ref{General Convergence Proposition 1} and~\ref{Kolmogorov-Chentsov Proposition} imply the second claim.

(ii) This assertion is an immediate application of Proposition~\ref{General SDE Proposition} in the case that $b=\underline{B} + R + B_{H}\dot{h}$ and $\sigma=\overline{B} + \Sigma$.

(iii) By Propositions~\ref{General Convergence Proposition 3},~\ref{General Convergence Proposition 1} and~\ref{Auxiliary Convergence Result 4} and Lemmas~\ref{Auxiliary Convergence Result 2} and~\ref{Auxiliary Convergence Result 3}, to establish~\eqref{General SDE Partition Limit Equation}, it suffices to show that there is $c_{2} > 0$ such that
\[
E\bigg[\max_{j\in\{0,\dots,k_{n}\}}\bigg|\int_{r}^{t_{j,n}}\big(\overline{B}(s,\null_{n}Y)-\overline{B}(\underline{s}_{n},\null_{n}Y)\big)\null_{n}\dot{W}_{s} - R(\underline{s}_{n},\null_{n}Y)\gamma_{n}(s)\,ds\bigg|^{2}\bigg] \leq c_{2}|\mathbb{T}_{n}|
\]
for each $n\in\mathbb{N}$. As $\partial_{x}\overline{B}$ is bounded, the existence of such a constant $c_{2}$ follows immediately from the decomposition~\eqref{Main Remainder Decomposition}, a combination of Proposition~\ref{Main Remainder Proposition 1} and Lemma~\ref{Main Remainder Lemma 2} with Lemma~\ref{Auxiliary Convergence Result 1} and an application of Proposition~\ref{Main Remainder Proposition 5}. Moreover, since $\sup_{n\in\mathbb{N}} E[\|\null_{n}Y\|_{\beta,r}^{p}] + E[\|Y\|_{\beta,r}^{p}] < \infty$ for all $\beta\in [0,1/2)$ and $p\geq 1$, the second assertion can now be inferred from Lemma~\ref{Hoelder Convergence Lemma}.
\end{proof}

To prove Theorem~\ref{Support Theorem} we require the following basic result on the support of image probability measures.

\begin{lemma}\label{General Support Lemma 1}
Let $(\tilde{\Omega},\tilde{\mathscr{F}},\tilde{P})$ be a probability space, $(S,\nu)$ be a metric space, $D\subset S$ and $Y:\tilde{\Omega}\rightarrow S$ be measurable such that $\tilde{P}\circ Y^{-1}$ is inner regular.
\begin{enumerate}[(i)]
\item Let $(Y_{n})_{n\in\mathbb{N}}$ be a sequence of $S$-valued measurable maps on $\tilde{\Omega}$ that converges in probability to $Y$. If $Y_{n}\in D$ a.s.~for all $n\in\mathbb{N}$, then $\mathrm{supp}(\tilde{P}\circ Y^{-1})\subset \overline{D}$.
\item Suppose that for each $y\in D$ there exists a sequence $(\tilde{P}_{y,n})_{n\in\mathbb{N}}$ of probability measures on $(\tilde{\Omega},\tilde{\mathscr{F}})$ such that $\tilde{P}_{y,n}\ll \tilde{P}$ for all $n\in\mathbb{N}$ and
\begin{equation}\label{General Support Inequality 1}
\inf_{n\in\mathbb{N}} \tilde{P}_{y,n}(\nu(Y,y)\geq \varepsilon) < 1
\end{equation}
for each $\varepsilon > 0$. Then $\overline{D}\subset\mathrm{supp}(\tilde{P}\circ Y^{-1})$.
\end{enumerate}
\end{lemma}

\begin{proof}
(i) Let $y\in\mathrm{supp}(\tilde{P}\circ Y^{-1})$ and $k\in\mathbb{N}$, then there exists $n_{k}\in\mathbb{N}$ such that $\tilde{P}(\nu(Y_{n},Y) > 1/(2k))$ $< \tilde{P}(\nu(Y,y) < 1/(2k))$ for all $n\in\mathbb{N}$ with $n\geq n_{k}$. By the triangle inequality,
\[
\tilde{P}(\nu(Y_{n},y)\geq 1/k) \leq \tilde{P}(\nu(Y_{n},Y) > 1/(2k)) + \tilde{P}(\nu(Y,y)\geq 1/(2k)) < 1
\]
for any such $n\in\mathbb{N}$. So, there is $\tilde{\omega}_{k}\in\tilde{\Omega}$ such that $y_{k}:=Y_{n_{k}}(\tilde{\omega}_{k})\in D$ and $\nu(y_{k},y) < 1/k$. As $k\in\mathbb{N}$ has been arbitrarily chosen, the resulting sequence $(y_{k})_{k\in\mathbb{N}}$ converges to $y$, which gives the claim.

(ii) By way of contradiction, assume that there are $y\in\overline{D}$ and $\varepsilon > 0$ such that $\tilde{P}(\nu(Y,y)\geq \varepsilon ) = 1$. Let $(y_{n})_{n\in\mathbb{N}}$ be a sequence in $D$ that converges to $y$ and choose $n_{\varepsilon}\in\mathbb{N}$ such that $\nu(y_{n_{\varepsilon}},y) < \varepsilon/2$. Then from
\[
\tilde{P}(\nu(Y,y)\geq \varepsilon) \leq \tilde{P}(\nu(Y,y_{n_{\varepsilon}}) > \varepsilon/2) + \tilde{P}(\nu(y_{n_{\varepsilon}},y) \geq \varepsilon/2)
\]
and $\tilde{P}_{y_{n_{\varepsilon}},n}\ll \tilde{P}$ it follows that $\tilde{P}_{y_{n_{\varepsilon}},n}(\nu(Y, y_{n_{\varepsilon}})\geq\varepsilon/2) = 1$ for each $n\in\mathbb{N}$. This, however, is a contradiction to~\eqref{General Support Inequality 1}.
\end{proof}

\begin{proof}[Proof of Theorem~\ref{Support Theorem}]
(i) Pathwise uniqueness follows from Lemma~\ref{General SDE Lemma 2} and the other two assertions are direct consequences of Proposition~\ref{General SDE Proposition}. 

(ii) For $h\in H_{r}^{1}([0,T],\mathbb{R}^{d})$ we set $F_{h}:=b - (1/2)\rho + \sigma\dot{h}$ and first check that $F_{h}$ satisfies conditions~\eqref{C.1}~and~\eqref{C.2}. Since $\sigma$ and $\partial_{x}\sigma$ are bounded, there is $c_{0}\geq 0$ such that $|\sigma|\vee |\rho|\leq c_{0}$. Then
\[
|F_{h}(t,x)|\leq c_{1}(1+|\dot{h}(t)|)(1 + \|x\|_{\infty})
\]
for all $t\in [r,T)$ and $x\in C([0,T],\mathbb{R}^{m})$ with $c_{1}:=3\max\{c,c_{0}\}$. Moreover, since $\sigma$ and $\partial_{x}\sigma$ are $d_{\infty}$-Lipschitz continuous, so is the map $\rho$. Let $\lambda_{0}\geq 0$ be a Lipschitz constant for $\rho$, then
\[
|F_{h}(t,x) - F_{h}(t,y)|\leq \lambda_{1}(1 + |\dot{h}(t)|)\|x-y\|_{\infty}
\]
for all $t\in [r,T)$ and $x,y\in C([0,T],\mathbb{R}^{m})$ with $\lambda_{1}:=2\max\{\lambda,\lambda_{0}\}$. Hence, an application of Proposition~\ref{General ODE Proposition} yields the first assertion.

Regarding the second claim, let us also choose $g\in H_{r}^{1}([0,T],\mathbb{R}^{d})$ and define $c_{2}:=2^{2}\max\{c_{0},\lambda_{1}\}^{2}\max\{1,T-r\}$. Then the above estimation shows that
\begin{equation*}
\|x_{g}^{t} - x_{h}^{t}\|_{H,r}^{2}\leq c_{2}\int_{r}^{t}|\dot{g}(s) - \dot{h}(s)|^{2} + (1 + |\dot{h}(s)|^{2})\|x_{g}^{s} - x_{h}^{s}\|_{H,r}^{2}\,ds
\end{equation*}
for given $t\in [r,T]$. By Gronwall's inequality, $\|x_{g} - x_{h}\|_{H,r}^{2}\leq c_{3}e^{c_{3}\|h\|_{H,r}^{2}}\|g - h\|_{H,r}^{2}$ with $c_{3}:=c_{2}\exp((T-r)c_{2})$, and the verification is complete.

{(iii) Let $N_{\alpha}$ be the $P$-null set of all $\omega\in\Omega$ such that $X(\omega)\notin C_{r}^{\alpha}([0,T],\mathbb{R}^{m})$, then $(N_{\alpha}^{c},\mathscr{F}\cap N_{\alpha}^{c},P_{|\mathscr{F}\cap N_{\alpha}^{c}})$ is a probability space and the probability measure
\begin{equation}\label{Support Image Measure}
\mathscr{B}(C_{r}^{\alpha}([0,T],\mathbb{R}^{m}))\rightarrow [0,1],\quad B\mapsto P(\{X\in B\}\cap N_{\alpha}^{c})
\end{equation}
is inner regular and its support agrees with the support of $P\circ X^{-1}$ in $C_{r}^{\alpha}([0,T],\mathbb{R}^{m})$. We note that the inner regularity can be inferred by using that $X$ belongs a.s.~to the separable closed linear subspace of all $x\in C_{r}^{\alpha}([0,T],\mathbb{R}^{m})$ such that 
\[
\lim_{\delta\downarrow 0} \sup_{s,t\in [r,T]:\,0 < |s-t|\leq \delta} \frac{|x(s)-x(t)|}{|s-t|^{\alpha}} = 0.
\]

Since~\eqref{Support Limit Equality 1} follows from Theorem~\ref{General SDE Convergence Theorem} by the choice $\underline{B}= b - (1/2)\rho$, $B_{H}=0$, $\overline{B}=\sigma$ and $\Sigma = 0$, Lemma~\ref{General Support Lemma 1} entails that the support of~\eqref{Support Image Measure} is included in the closure of $\{x_{h}\,|\,h\in H_{r}^{1}([0,T],\mathbb{R}^{d})\}$ with respect to $\|\cdot\|_{\alpha,r}$.

Next, let $h\in H_{r}^{1}([0,T],\mathbb{R}^{d})$ and for any $n\in\mathbb{N}$ we note that the non-anticipative product measurable map $[r,T]\times C([0,T],\mathbb{R}^{d})\rightarrow\mathbb{R}^{d}$, $(t,x)\mapsto \dot{L}_{n}(x)(t)$ satisfies $|\dot{L}_{n}(x)(t)|\leq 2c_{\mathbb{T}}|\mathbb{T}_{n}|^{-1}\|x\|_{\infty}$ for all $t\in [r,T]\backslash\mathbb{T}_{n}$ and $x\in C([0,T],\mathbb{R}^{d})$ and the map $C([0,T],\mathbb{R}^{d})\rightarrow\mathbb{R}^{d}$, $x\mapsto \dot{L}_{n}(x)(t)$ is linear for each $t\in [r,T]\backslash \mathbb{T}_{n}$.

From these facts it follows that for every $x\in C([0,T],\mathbb{R}^{d})$ there is a unique mild solution $y_{h,n,x}\in C([0,T],\mathbb{R}^{d})$ to the following ordinary integral equation with running value condition:
\[
y_{h,n,x}(t) = x(t) - \int_{r}^{r\vee t}\dot{h}(s) - \dot{L}_{n}(y_{h,n,x})(s)\,ds\quad\text{for $t\in [0,T]$.}
\]
Moreover, the map $C([0,T],\mathbb{R}^{d})\rightarrow C([0,T],\mathbb{R}^{d})$, $x\mapsto y_{h,n,x}$ is Lipschitz continuous on bounded sets and in particular, Borel measurable. These considerations allow us to define a process $\null_{h,n}W\in\mathscr{C}([0,T],\mathbb{R}^{d})$ by $\null_{h,n}W_{t}:=y_{h,n,W}(t)$.

Given this constructed process, it follows from Lemma~\ref{Doleans-Dade Exponential Lemma} that we obtain a martingale $\null_{h,n}Z\in\mathscr{C}([0,T],\mathbb{R})$ by requiring that $\null_{h,n}Z^{r}=1$ and
\[
\null_{h,n}Z_{t} = \exp\bigg(\int_{r}^{t}\dot{h}(s)' - \dot{L}_{n}(\null_{h,n}W)(s)'\,dW_{s} - \frac{1}{2}\int_{r}^{t}|\dot{h}(s) - \dot{L}_{n}(\null_{h,n}W)(s)|^{2}\,ds\bigg)
\]
for every $t\in [r,T]$ a.s.~and we may define a probability measure $P_{h,n}$ on $(\Omega,\mathscr{F})$ that is equivalent to $P$ by $P_{h,n}(A):=E[\mathbbm{1}_{A}\null_{h,n}Z_{T}]$. According to the second part of Lemma~\ref{General Support Lemma 1}, if
\begin{equation}\label{Support Limit Inequality}
\inf_{n\in\mathbb{N}} P_{h,n}(\{\|X - x_{h}\|_{\alpha,r}\geq\varepsilon\}\cap N_{\alpha}^{c}) < 1\quad\text{for each $\varepsilon > 0$,}
\end{equation}
then the closure of $\{x_{g}\,|\,g\in H_{r}^{1}([0,T],\mathbb{R}^{d})\}$ with respect to $\|\cdot\|_{\alpha,r}$ is included in the support of~\eqref{Support Image Measure}. Now Girsanov's theorem implies that for each $n\in\mathbb{N}$ the process $\null_{h,n}W$ is a $d$-dimensional $(\mathscr{F}_{t})_{t\in [0,T]}$-Brownian motion under $P_{h,n}$ and $X$ is a strong solution to~\eqref{Girsanov SDE} under $P_{h,n}$.

Hence, by using uniqueness in law, an application of Theorem~\ref{General SDE Convergence Theorem} in the case that $\underline{B}=b$, $B_{H}=\sigma$, $\overline{B}=-\sigma$ and $\Sigma=\sigma$ gives~\eqref{Support Limit Equality 2}. As this readily implies~\eqref{Support Limit Inequality}, the proof is complete.}
\end{proof}

\end{document}